\documentclass[12pt]{article}
\usepackage{amssymb}
\usepackage{amsmath}
\usepackage{latexsym}
\usepackage{mathrsfs}
\usepackage{amsthm}
\usepackage{graphicx}
\usepackage{caption}
\usepackage{hyperref}
\usepackage{color}
\usepackage{enumerate}
\allowdisplaybreaks[1]
\usepackage{comment}

\numberwithin{equation}{section}
\newtheorem{theorem}[equation]{Theorem}
\newtheorem{proposition}[equation]{Proposition}
\newtheorem{defn}[equation]{defn}
\newtheorem{rem}[equation]{rem}
\newtheorem{lemma}[equation]{Lemma}
\newtheorem{corollary}[equation]{Corollary}
\usepackage{comment}
\begin{document}

\title{Complementary asymptotically sharp estimates for eigenvalue means of Laplacians}


\author{Evans M. Harrell II \footnote{School of Mathematics, Georgia Institute of Technology, Atlanta GA 30332-0160, USA} , Luigi Provenzano \footnote{Universit\`a degli Studi di Padova, Dipartimento di Matematica, Via Trieste 63, 35121 Padova, Italy} and Joachim Stubbe \footnote{EPFL, SB Institute of Mathematics, Station 8, CH-1015 Lausanne, Switzerland}}





\maketitle

\noindent
{\bf Abstract:}
We present asymptotically sharp inequalities, containing a second term, for the Dirichlet and Neumann eigenvalues of the Laplacian on a domain, which are complementary to the familiar Berezin-Li-Yau and Kr\"oger inequalities in the limit as the eigenvalues tend to infinity.

We accomplish this in the framework of the Riesz mean $R_1(z)$ of the eigenvalues by applying the averaged variational principle with families of test functions that have been corrected for boundary behaviour.
\vspace{11pt}

\noindent
{\bf Keywords:}  Dirichlet Laplacian, Neumann Laplacian, Semiclassical bounds for eigenvalues, Averaged variational principle, Tubular neighbourhood, Distance to the boundary.

\vspace{6pt}
\noindent
{\bf 2010 Mathematics Subject Classification:} 35P15, 35P20, 47A75, 58J50.




\section{Introduction}

According to a conjecture made by P\'olya in 1961, the Weyl estimate of large eigenvalues should be a strict lower bound for 
each of the Dirichlet eigenvalues of a domain, and an upper bound for the Neumann eigenvalues.  P\'olya's still open conjecture has inspired legions of articles,
notably those in the tradition of Berezin \cite{Berezin} and of Li and Yau \cite{LiYau}, who proved an averaged version of the conjecture in the Dirichlet case:
\begin{equation*}
\frac{1}{k}\sum_{j=1}^k\lambda_j\geq\frac{d}{d+2}C_d\left(\frac{k}{|\Omega|}\right)^{\frac{2}{d}},
\end{equation*}
and of Kr\"oger  \cite{Kro}, who did the same for the Neumann case  (see also \cite{Lap1997}):
\begin{equation}\label{kroger_0}
\frac{1}{k}\sum_{j=1}^k\mu_j\leq\frac{d}{d+2}C_d\left(\frac{k}{|\Omega|}\right)^{\frac{2}{d}}.
\end{equation} 
Here
\begin{equation}\label{Cd}
C_d:=4\pi^2\omega_d^{-\frac{2}{d}}
\end{equation}
and  $\omega_d:=\Gamma(1+d/2)^{-1}\pi^{d/2}$ is the volume of the unit ball in $\mathbb R^d$. The so-called ``classical constant'' $C_d$ is related to the Weyl asymptotic relation satisfied by  Neumann and Dirichlet eigenvalues, namely
$$
\lim_{j\rightarrow+\infty}\frac{\lambda_j}{j^{\frac{2}{d}}}=\lim_{j\rightarrow+\infty}\frac{\mu_j}{j^{\frac{2}{d}}}=\frac{C_d}{|\Omega|^{\frac{2}{d}}},
$$
which shows that the Berezin-Li-Yau and Kr\"oger bounds are asymptotically sharp.

A paper by Melas \cite{Mel} opened the way to improving the Berezin-Li-Yau inequality with a
lower-order correction by incorporating more information about eigenfunctions, and a correction with the expected order 
was later obtained by Weidl \cite{wei0} (see also Geisinger, Laptev, and Weidl \cite{GeLaWe}).  More recently, an
article by two of us \cite{harrell_stubbe_18} improved Weyl-sharp inequalities for eigenvalues of the Laplacian on domains, with both Dirichlet and 
Neumann boundary conditions.
In \cite{harrell_stubbe_18} the Berezin-Li-Yau and Kr\"oger inequalities were replaced by two-term expressions of the right orders, with tight constants.\\
The two-term expressions in \cite{harrell_stubbe_18} were derived in the framework of Riesz means, which have come to be recognized as an efficient way to understand Weyl asymptotics.  By definition,
$R_\sigma(z) :=
\sum_{j}(z-\lambda_{j})_{+}^\sigma$, or respectively $\sum_{j}(z-\mu_{j})_{+}^\sigma$.
(Here $x_{+}$ denotes the positive part of $x$.)
Note that in the way that Riesz means are defined, Berezin-Li-Yau becomes an {\em upper} bound for an expression involving the Dirichlet eigenvalues, and Kr\"oger a {\em lower} bound with the Neumann eigenvalues.  
It is striking that the proofs of the Berezin-Li-Yau and Kr\"oger inequalities use similar ingredients, including the Fourier transform, 
but that they are arranged in different ways.

In this work we show that in some circumstances the situation can be reversed, so that there is a kind of Berezin-Li-Yau {\em upper} bound for the {\em Neumann} Riesz means (i.e., a lower bound for eigenvalue averages) and a kind of Kr\"oger {\em lower} bound for {\em Dirichlet} Riesz means (i.e., an upper bound for eigenvalue averages).  Reversing the inequalities certainly requires lower-order correction terms, which, as will be seen, include information about the boundary of the domain.  As in \cite{harrell_stubbe_18} an essential tool will be the averaged variational principle first introduced in \cite{HaSt14} (see also \cite{EHIS16}), which gives an efficient derivation of Kr\"oger's inequality and has been used to
derive various other upper bounds for averages of eigenvalues. The averaged variational principle applies most directly to $R_1$, which is easily connected to averages of eigenvalues via the Legendre transform:
\[
{\mathcal L}\left[R_1\right](w) = \left(w -  \left[w\right]\right) \lambda_{\left[w\right]+1}
+  \sum_{j=1}^{\left[w\right]}\lambda_j.
\]
We observe that similar results can be found in \cite{Kro2}, where, however, bounds are given directly for sums under quite technical assumptions on the domain, and the constants are implicit and don't show an immediate dependence on the geometry of the domain and of its boundary. We also refer to \cite{stri} where Kr\"oger's results are extended to the case of homogeneous spaces.

The averaged variational principle \cite{HaSt14,EHIS16} applies to any self-adjoint operator $H$
on a Hilbert space $(\mathcal{H},\langle,\rangle_{\mathcal{H}})$
having purely discrete spectrum consisting of eigenvalues $0\leq \lambda_1\leq \lambda_2\leq\ldots\leq\lambda_j\leq\ldots$.  We recall it here in a form adapted to our purposes.
Let the eigenvectors corresponding to $\lambda_j$ be written as $u_j$, let
$Q_H(f,f)$ be the quadratic form associated with $H$, defined on the form domain $\mathcal{Q}_H$, and let
$\mathcal{D}_H$ be the operator domain of $H$.
Note that $Q_H(f,f)=\langle Hf,f\rangle_{\mathcal H}$ whenever $f\in \mathcal{D}_H$.
Let $P_k$ be the spectral projector associated with the first $k$ eigenvalues. Then for any $f\in \mathcal{Q}_H$, according to the standard variational principle,

\begin{equation}\label{var-principle}
 \lambda_{k+1}\big(\langle f,f\rangle_{\mathcal{H}}-\langle P_kf,P_kf\rangle_{\mathcal{H}}\big)\leq Q_H(f,f)-Q_H(P_kf,P_kf).
\end{equation}

Suppose now that a family of functions $f_\xi \in \mathcal{Q}_H$ is indexed by
$\xi \in M$ in a measure space
$(M,\Sigma,\mu)$, such that the {\em tight-frame} condition holds, that is, for any $\phi \in \mathcal{H}$, 
\[
\int_M|\left\langle\phi, f_\xi\right\rangle_{\mathcal H}|^2 d\mu(\xi)= C \|\phi\|_{\mathcal H}^2.
\]
Let $\mu_0$ be another measure on $(M,\Sigma,\mu)$ such that $\displaystyle \int_{M}g(\xi)\; d\mu_0\leq \int_{M}g(\xi)\; d\mu$ for any nonnegative real-valued function $g$ on $M$. 
Then after integrating \eqref{var-principle} over $\xi\in M$ with respect to $\mu_0$ and reorganizing the terms we obtain
\begin{equation*}
\sum_{j=1}^k(\lambda_{k+1}-\lambda_j)\int_M|\langle f_{\xi},u_j\rangle_{\mathcal H}|^2d\mu_0(\xi)\geq\int_M\left(\lambda_{k+1}\|f_\xi\|^2 - Q_H(f_\xi,f_\xi)\right) d\mu_0(\xi),
\end{equation*}
which implies that
\begin{equation}\label{averaged-var-principle-Rieszmean2_0}
\sum_{j=1}^k(\lambda_{k+1}-\lambda_j)\int_M|\langle f_{\xi},u_j\rangle_{\mathcal H}|^2d\mu(\xi)\geq\int_M\left(\lambda_{k+1}\|f_\xi\|^2 - Q_H(f_\xi,f_\xi)\right) d\mu_0(\xi)
\end{equation}
since the terms on left side are nonnegative.

Since we may replace $\lambda_{k+1}$  in \eqref{averaged-var-principle-Rieszmean2_0}
with any $z\in[\lambda_k,\lambda_{k+1}]$ and we can choose $d\mu_0(\xi)=1_{\left\{z\|f_\xi\|^2 - Q_H(f_\xi,f_\xi)\geq 0\right\}}d\mu(\xi)$ ($1_A$ denoting the characteristic function of $A\subset M$), we have
\begin{equation}\label{averaged-var-principle-Rieszmean2}
R_1(z) = \sum_{j}(z - \lambda_j)_+\ge \frac 1 C \int_M\left(z\|f_\xi\|_{\mathcal H}^2 - Q_H(f_\xi,f_\xi)\right)_+ d\mu(\xi).
\end{equation}
Inequalities \eqref{averaged-var-principle-Rieszmean2_0} and \eqref{averaged-var-principle-Rieszmean2} are the forms of the averaged variational principle which we
shall exploit throughout this paper.

When we consider unbounded operators, typically $M$ and $f_{\xi}$ are such that
\begin{equation*}
    \int_{M}Q_H(P_kf_{\xi},P_kf_{\xi})\,d\mu(\xi)=\int_{M}\sum_{j=1}^k\lambda_j|\langle f_{\xi},u_j\rangle_{\mathcal{H}}|^2\,d\mu(\xi)
\end{equation*}
and
\begin{equation*}
 \int_{M}\langle P_kf_{\xi},P_kf_{\xi}\rangle_{\mathcal{H}}\,d\mu(\xi)=\int_{M}\sum_{j=1}^k|\langle f_{\xi},u_j\rangle_{\mathcal{H}}|^2\,d\mu(\xi)
\end{equation*}
are easily computable and finite, but
\begin{equation*}
    \int_{M}Q_H(f_{\xi},f_{\xi})\,d\mu(\xi)=+\infty, \quad \int_{M}\langle f_{\xi},f_{\xi}\rangle_{\mathcal{H}}\,d\mu(\xi) = +\infty,
\end{equation*}
so therefore we need to choose an appropriate $\mu_0$ in \eqref{averaged-var-principle-Rieszmean2_0}.

We also remark here that the variational principle can be seen as a {\it trace inequality}. In fact, let $H$ be as before and $P$ be an orthogonal projector commuting with $H$ such that $(1-P)(H-z)(1-P)\geq 0$ for some $z\in\mathbb{R}$. Then (by the standard variational principle) for any $f\in \mathcal{D}_H$:
\begin{equation*}
  \langle f,P(H-z)P f\rangle \leq \langle f,(H-z) f\rangle.
\end{equation*}
Let $Q$ be another linear operator. By hypothesis
\begin{equation*}
  \text{tr}( Q(1-P)(H-z)(1-P)Q)\geq 0,
\end{equation*}
which is trivially equivalent to
\begin{equation*}
  \text{tr}( QP(H-z)PQ)\leq \text{tr}( Q(H-z)Q).
\end{equation*}
provided that all the traces above are finite, which is true 
in particular when $Q$ has finite-dimensional range). We may choose $Q$ to be an orthogonal projector, in particular the one-dimensional projector onto $f$, that is $\displaystyle Qu=\frac{\langle f,u\rangle f }{\langle f,f\rangle}$ for $f\neq 0$. Obviously by this we recover the usual variational inequality for $f$.

As already mentioned, reversing the Berezin-Li-Yau and Kr\"oger inequalities requires lower-order correction terms, which we want to compare with the semiclassical behaviour of the eigenvalues.
With this in mind we recall that if $\Omega$ is a bounded domain with a sufficiently smooth boundary (for example, if $\Omega$ is Lipschitz in the case of Dirichlet eigenvalues or if it is of class $C^{1}$ in case of Neumann eigenvalues, see \cite{frank,frank2,frank3}),
then the following two-term asymptotic formula for the average of the first $k$ eigenvalues holds:
\begin{equation}\label{weyl_D}
\frac{1}{k}\sum_{j=1}^k\lambda_j=\frac{d}{d+2}C_d\left(\frac{k}{|\Omega|}\right)^{\frac{2}{d}}+\frac{1}{2(d+1)}\frac{C_d^{\frac{d+1}{2}}}{C_ {d-1}^{\frac{d-1}{2}}}\frac{|\partial\Omega|}{|\Omega|}\left(\frac{k}{|\Omega|}\right)^{\frac{1}{d}}+o(k^{\frac{1}{d}})
\end{equation}
as $k\rightarrow +\infty$, for Dirichlet boundary conditions, and
\begin{equation}\label{weyl_N}
\frac{1}{k}\sum_{j=1}^k\mu_j=\frac{d}{d+2}C_d\left(\frac{k}{|\Omega|}\right)^{\frac{2}{d}}-\frac{1}{2(d+1)}\frac{C_d^{\frac{d+1}{2}}}{C_ {d-1}^{\frac{d-1}{2}}}\frac{|\partial\Omega|}{|\Omega|}\left(\frac{k}{|\Omega|}\right)^{\frac{1}{d}}+o(k^{\frac{1}{d}})
\end{equation}
as $k\rightarrow +\infty$, for Neumann boundary conditions, where $|\partial\Omega|$ denotes the $d-1$ dimensional Hausdorff measure of the boundary.

We declare now that henceforth we consider only $d\geq 2$. In fact for $d=1$ all the eigenvalues are explicitly known.

The present paper is organized as follows. In Section \ref{sec:1.5} we state the main results concerning Dirichlet eigenvalues. In particular we state a general theorem on lower bounds for Riesz means of Dirichlet eigenvalues and upper bounds for averages, namely Theorem \ref{dirichlet_laplacian_thm_general}, which holds for all bounded domains in $\mathbb R^d$. In particular the bounds provided by Theorem \ref{dirichlet_laplacian_thm_general} depend on the choice of a test function $\phi\in H^1_0(\Omega)\cap L^{\infty}(\Omega)$. As a corollary, still in Section \ref{sec:1.5}, we get lower bounds for the partition function (Corollary \ref{dirichlet_bounds_spec_fcn}), and two-sided bounds for single eigenvalues (Corollary \ref{quadratic-eigenvalue-inequality-general}), depending on $\phi$. In Subsection \ref{sec:1.5.1} we provide explicit bounds for bounded domains without further regularity assumptions (Theorem \ref{simple_bound}) and for convex domains (Theorem \ref{thm_convex_first_eigenvalues}) by choosing a suitable test function $\phi$ in Theorem \ref{dirichlet_laplacian_thm_general}. In Subsection \ref{sec:1.5.2} we present explicit bounds under the assumption that the Minkowski content relative to $\Omega$ equals the Hausdorff measure of the boundary (Theorem \ref{main_thm_evsums-DirichletLaplacian} and Corollary \ref{thm_bounds_part_func_regular} ). In Subsection \ref{sec:1.5.3} we present explicit bounds under the assumption that the boundary is of class $C^2$ (Theorem \ref{smooth_1}) and additionally that it is mean convex (Corollary \ref{smooth_convex}). Subsection \ref{sec:1.5.4} contains more explicit estimates for planar sets (Theorem \ref{planar_thm}). The bounds in Subsections \ref{sec:1.5.2}, \ref{sec:1.5.3} and \ref{sec:1.5.4} are asymptotically sharp and present a second term which coincides with the second term of the semiclassical expansion \eqref{weyl_D} up to a dimensional constant. Moreover we have explicit geometric control of the remainder term.

In Section \ref{sec:1.6} we state the main results concerning Neumann eigenvalues. In particular, we shall present improvements of classical upper bounds for Neumann eigenvalues (Theorem \ref{refinement_kroger_thm}, see also \cite{harrell_stubbe_18}) as well as lower bounds for averages and upper bounds for Riesz means (Theorem \ref{thm_bounds_neumann_eigenvalues}) of Neumann eigenvalues for domains of class $C^2$ by means of the averaged variational principle. These bounds are asymptotically sharp and present a second term which coincides with the second term of the corresponding semiclassical expansion \eqref{weyl_N} up to a dimensional constant. Moreover we have explicit geometric control of the remainder term.

The proofs of the results stated in Sections \ref{sec:1.5} and \ref{sec:1.6} are contained in Sections \ref{sec:2} and \ref{sec:3}, respectively.

In Appendix \ref{sec:4} we  present some final remarks. In particular, in Appendix \ref{sub:4:1} we  show how to recover Theorem \ref{thm_bounds_neumann_eigenvalues} alternatively through a generalization to any Laplace eigenfunction of the method of Berezin-Li-Yau (\cite{Berezin,LiYau}). In Appendix \ref{sub:4:2} we  show how to obtain asymptotically Weyl-sharp upper and lower bounds for single Dirichlet and Neumann eigenvalues from bounds on averages.




\section{The main results for Dirichlet eigenvalues}\label{sec:1.5}

In this section we present bounds for averages of eigenvalues of $-\Delta^D_{\Omega}$ obtained from an application of the averaged variational principle. Here $-\Delta^D_{\Omega}$ denotes the self-adjoint realization of the (nonnegative) Laplacian on $\Omega$ with Dirichlet boundary conditions. Here and in what follows we denote by $\left\{\lambda_j\right\}_{j=1}^{\infty}$ the set of the (positive) eigenvalues of $-\Delta^D_{\Omega}$ and by $\left\{u_j\right\}_{j=1}^{\infty}$ the  corresponding orthonormal set in $L^2(\Omega)$ of eigenfunctions.

We will develop a general approach which yields upper bounds for Dirichlet eigenvalues for a quite wide class of domains (Theorem \ref{dirichlet_laplacian_thm_general}). Then, under more regularity assumptions on the domain, the bounds given by Theorem \ref{dirichlet_laplacian_thm_general}  can be made more explicit (see Theorems \ref{simple_bound}, \ref{thm_convex_first_eigenvalues}, \ref{main_thm_evsums-DirichletLaplacian}, \ref{smooth_1} and \ref{planar_thm}, and Corollaries \ref{thm_bounds_part_func_regular} and \ref{smooth_convex}).  In particular we provide two-term bounds showing the correct asymptotic behaviour. We also refer to  \cite{Kro2} and \cite{Lap1997} for related results.

In the following, we denote by $\|\cdot\|_{2}$ and $\|\cdot\|_{\infty}$ the standard norms on $L^2(\Omega)$ and $L^{\infty}(\Omega)$ respectively. We also denote by $1_A$ the characteristic function of $A\subseteq\mathbb R^d$. We write $\sum_j$ to indicate the sum over all positive integers $j$.

Applying the averaged variational principle \eqref{averaged-var-principle-Rieszmean2_0} with test functions of the form $f_{\xi}(x)=(2\pi)^{-d/2}e^{i\xi\cdot x}\phi(x)$, with $\phi(x)\in H^1_0(\Omega)\cap L^{\infty}(\Omega)$, we obtain the following theorem.

\begin{theorem}\label{dirichlet_laplacian_thm_general}
Let $\Omega$ be a bounded domain in $\mathbb R^d$. Then for any $\phi\in H^1_0(\Omega)\cap L^{\infty}(\Omega)$ and $z>0$ the following inequality holds
\begin{equation}\label{Riesz-mean-ineq-DirichletLaplacian}
  \sum_{j}(z-\lambda_{j})_{+}\|\phi u_j\|_2^2\geq \frac{2}{d+2}\,(2\pi)^{-d}\omega_d||\phi||_2^2\left(z-\frac{||\nabla\phi||_2^2}{||\phi||_2^2}\right)_{+}^{\frac{d}{2}+1}.
\end{equation}
Moreover, for all positive integers $k$
\begin{equation}\label{evsums-DirichletLaplacian1}
 \frac1{k}\sum_{j=1}^k\lambda_{j}\leq \frac{||\nabla\phi||_2^2}{||\phi||_2^2} +\frac{d}{d+2}\,C_d\left(\frac{k}{|\Omega|}\right)^{\frac{2}{d}}\rho(\phi)^{-2/d},
\end{equation}
where $\displaystyle \rho(\phi):=\frac{||\phi||_2^2}{|\Omega|\cdot||\phi||_{\infty}^2}<1$.
\end{theorem}

\begin{rem}
The right side of inequality \eqref{evsums-DirichletLaplacian1} relates the upper bound to the semiclassical behaviour of the average of the first $k$ eigenvalues, which is, by the work of Berezin, Li and Yau (\cite{Berezin,LiYau}), a lower bound for the average.
\end{rem}

A lower bound for the partition function (the trace of the heat kernel) is obtained by Laplace transforming inequality \eqref{Riesz-mean-ineq-DirichletLaplacian}. We have the following corollary.

\begin{corollary}\label{dirichlet_bounds_spec_fcn}
Let $\Omega$ be a bounded domain in $\mathbb R^d$. Then for any $\phi\in H^1_0(\Omega)\cap L^{\infty}(\Omega)$ and $t>0$,
\begin{equation}\label{Partition-function-inequality}
  \sum_{j=1}^{\infty}e^{-\lambda_{j}t}\|\phi u_j\|_2^2\geq (4\pi t)^{-d/2}||\phi||_2^2e^{-\frac{||\nabla\phi||_2^2}{||\phi||_2^2}\,t}.
\end{equation}
Moreover
\begin{equation}\label{part-fct-estimate-small-times}
  \sum_{j=1}^{\infty}e^{-\lambda_{j}t}\geq \frac{|\Omega|}{(4\pi t)^{\frac{d}{2}}}-\frac{1}{(4\pi t)^{\frac{d}{2}}}\cdot\frac{||\nabla\phi||_2^2\,t+|\Omega|\cdot||\phi||_{\infty}^2-||\phi||_2^2}{||\phi||_{\infty}^2}.
\end{equation}
\end{corollary}

From Theorem \ref{dirichlet_laplacian_thm_general} we also deduce ``intrinsic'' bounds on single eigenvalues, which are stated in the following corollary.

\begin{corollary}\label{quadratic-eigenvalue-inequality-general}
Let $\Omega$ be a bounded domain in $\mathbb R^d$. Let $\phi\in H^1_0(\Omega)\cap L^{\infty}(\Omega)$ and suppose that $\displaystyle \lambda_k\geq \frac{||\nabla\phi||_2^2}{||\phi||_2^2}$. Then the bounds
\begin{equation}\label{AVP-eigenvalue-bound}
  C_d\left(\frac{k}{|\Omega|}\right)^{\frac{2}{d}}\rho(\phi)^{-2/d}x_{-}\leq \lambda_k\leq \lambda_{k+1}\leq  C_d\left(\frac{k}{|\Omega|}\right)^{\frac{2}{d}}\rho(\phi)^{-2/d}x_{+}
\end{equation}
hold true, where
\begin{equation*}
  x_{\pm}=1\pm \sqrt{1-\frac{\frac{d+2}{d}\frac1{k}\sum_{j=1}^k\lambda_{j}-\frac{||\nabla\phi||_2^2}{||\phi||_2^2}}{ C_d\left(\frac{k}{|\Omega|}\right)^{\frac{2}{d}}\rho(\phi)^{-2/d}}}.
\end{equation*}
\end{corollary}

\begin{rem}\label{rem_test}
We note that bounds \eqref{AVP-eigenvalue-bound} become asymptotically Weyl-sharp if we choose the test function $\phi$ in a suitable way, namely $\phi=\phi_k$ such that $\|\phi_k\|_2^2\sim |\Omega|$, $\|\nabla\phi_k\|_2^2\sim k^{1/d}$ and $\|\phi_k\|_{\infty}^2\sim 1$ as $k\rightarrow +\infty$. Suitable choices are described in Subsections \ref{sec:1.5.1}, \ref{sec:1.5.2}, \ref{sec:1.5.3} and \ref{sec:1.5.4}; however we
perform explicit computation mainly for sums. Explicit Weyl-sharp bounds for single eigenvalues can be obtained from \eqref{AVP-eigenvalue-bound} in a similar way. (See also Appendix \ref{sub:4:2}.)
\end{rem}

Theorem \ref{dirichlet_laplacian_thm_general} and its corollaries are then applied to domains with more and more regularity requirements. This leads to more explicit bounds by choosing suitable cut-off functions $\phi$ in \eqref{evsums-DirichletLaplacian1} and \eqref{part-fct-estimate-small-times}.




\subsection{Upper bounds with no restriction on the regularity of the boundary}\label{sec:1.5.1}

A first explicit bound can be obtained by choosing $\phi=u_1$, where $u_1$ is the eigenfunction associated with the first Dirichlet eigenvalue $\lambda_1$ of $\Omega$. Here and in what follows we shall denote by $r_{\Omega}$ the inradius of a domain $\Omega$, that is,
\begin{equation}\label{inradius}
r_{\Omega}:=\underset{x\in\Omega}{\max}\underset{y\in\partial\Omega}{\min}|x-y|.
\end{equation}

We can now state the following theorem.

\begin{theorem}\label{simple_bound}
Let $\Omega$ be a bounded domain in $\mathbb R^d$. Then for all $z\geq\frac{d+2}{2r_{\Omega}^2}$,
\begin{equation}\label{avp-DirichletLplacian3}
  \frac{d\lambda_1(B)^{\frac{d-1}{2}}}{|J_{\frac{d}{2}}(\sqrt{\lambda_1(B)})|}\,r_{\Omega}^{-d}\sum_{j}(z-\lambda_{j})_+\geq \frac{1}{d+2}(z-\lambda_{1})_+^{\frac{d}{2}+1},
\end{equation}
where $B$ denotes the unit ball in $\mathbb R^d$, $\lambda_1(B)$ is the first Dirichlet eigenvalue of $B$, $J_{\nu}$ denotes the Bessel function of the first kind and order $\nu$, $\lambda_1$ is the first Dirichlet eigenvalue of $\Omega$, and $r_{\Omega}$ is the inradius of $\Omega$.

Moreover, for all positive integers $k$,
\begin{equation}\label{evsums-DirichletLaplacian15_v2}
\frac1{k}\sum_{j=1}^k(\lambda_{j}-\lambda_1)\leq \frac{d}{d+2}C_d\left(\frac{k}{|\Omega|}\right)^{\frac{2}{d}}\cdot\left(\frac{2d\omega_d|\Omega|}{J_{\frac{d}{2}}(\sqrt{\lambda_1(B)})}\right)^{\frac{2}{d}}\cdot\frac{\lambda_1(B)^{\frac{d-1}{d}}}{4\pi^2r_{\Omega}^2}.
\end{equation}
\end{theorem}

We remark that the bound \eqref{evsums-DirichletLaplacian15_v2} depends on $\lambda_1$. We refer to Remarks \ref{remark1} and \ref{remark2} for a discussion of inequality \eqref{evsums-DirichletLaplacian15_v2}. We also mention \cite{freitas_ub} and references therein for a discussion of sharp upper bounds for the first Dirichlet eigenvalue on convex domains in terms of the diameter and the inradius.

In the case of convex sets we can choose $\phi$ in a more efficient way and obtain bounds which depend only on $|\Omega|$ and $|\partial\Omega|$. We have the following

\begin{theorem}\label{thm_convex_first_eigenvalues}
If $|\Omega|$ is a bounded convex set in $\mathbb R^d$ and $z\geq \frac{d+2}{2r_{\Omega}^2}$, then
\begin{equation}\label{evriesz-DirichletLaplacian15_v30-NEW}
  \sum_{j}(z-\lambda_j)_{+}\geq \frac{2}{d+2}\,(2\pi)^{-d}\omega_d |\Omega|z^{\frac{d}{2}+1}-2\sqrt{\frac{2}{d+2}}(2\pi)^{-d}\omega_d |\partial\Omega|z^{\frac{d}{2}+\frac{1}{2}}.
\end{equation}
Moreover, for all positive integers $k$,
\begin{equation}\label{evsums-DirichletLaplacian15_v30-NEW}
  \frac{1}{k}\sum_{j=1}^{k}\lambda_j\leq \frac{d}{d+2}\,C_d\left(\frac{k}{|\Omega|}\right)^{\frac{2}{d}}+2\left(\frac{2}{d+2}\,C_d\right)^{\frac{1}{2}}\left(\frac{k}{|\Omega|}\right)^{\frac{1}{d}}\frac{|\partial \Omega|}{|\Omega|}+\frac{4|\partial \Omega|^2}{|\Omega|^2}.
\end{equation}
Finally, for $\lambda_1$ there is the simnple bound
\begin{equation}\label{planar_convex_lambda12}
\lambda_1\leq 4\frac{|\partial\Omega|^2}{|\Omega|^2}.
\end{equation}
\end{theorem}

\begin{rem}
We can compare the bound \eqref{evriesz-DirichletLaplacian15_v30-NEW} to the well-known asymptotic expansion
\begin{equation*}
  \sum_{j}(z-\lambda_j)_{+}\sim \frac{2}{d+2}\,(2\pi)^{-d}\omega_d |\Omega|z^{\frac{d}{2}+1}-\frac{1}{4}\frac{2}{d+1}\,(2\pi)^{1-d}\omega_{d-1} |\partial\Omega|z^{\frac{d}{2}+\frac{1}{2}}
\end{equation*}
as $z$ tends to infinity. We note that the ratio
\begin{equation*}
  \frac{2\sqrt{\frac{2}{d+2}}(2\pi)^{-d}\omega_d }{\frac{1}{4}\,(2\pi)^{1-d}\omega_{d-1}}
\end{equation*}
is an increasing function of $d$ and tends to $\displaystyle 4\pi^{-1/2}$ as $d$ tends to infinity.  Therefore
\begin{equation*}
  \frac{3\sqrt{2}}{2}\leq \frac{2\sqrt{\frac{2}{d+2}}(2\pi)^{-d}\omega_d }{\frac{1}{4}\frac{2}{d+1}\,(2\pi)^{1-d}\omega_{d-1}}\leq  4\pi^{-1/2},
\end{equation*}
or, approximately,
\begin{equation*}
  2.121\leq \frac{2\sqrt{\frac{2}{d+2}}(2\pi)^{-d}\omega_d }{\frac{1}{4}\frac{2}{d+1}\,(2\pi)^{1-d}\omega_{d-1}}\leq  2.257.
\end{equation*}
\end{rem}

\begin{rem}
Consider \eqref{evsums-DirichletLaplacian15_v30-NEW} for the cartesian product $\Omega=\Omega'\times (0,L)$ with $\Omega'\subset\mathbb{R}^{d-1}$ convex, as $L$ goes to infinity. Let $\lambda_1(\Omega'), \lambda_2(\Omega')$ denote the first two Dirichlet eigenvalues of $\Omega'$. The low lying Dirichlet eigenvalues of $\Omega$, $\lambda_j$, are given by
\begin{equation*}
  \lambda_j=\lambda_1(\Omega')+\frac{\pi^2j^2}{L^2}
\end{equation*}
provided $\displaystyle \frac{\pi^2j^2}{L^2}< \lambda_2(\Omega')-\lambda_1(\Omega')$. Let $k$ be any index such that this inequality holds and define $\displaystyle\kappa=\frac{k}{L|\Omega'|}$. 
We keep $\kappa$ fixed when $L$ tends to infinity. Then, from inequality \eqref{evsums-DirichletLaplacian15_v30-NEW} we obtain, as $L\rightarrow+\infty$
\begin{multline}\label{ineq-8}
 \lambda_1(\Omega')\leq\frac{1}{k}\sum_{j=1}^k\lambda_j= \lambda_1(\Omega')+\frac{\pi^2|\Omega'|^2\kappa^2}{3}\\
\leq \frac{d}{d+2}\,C_d\kappa^{\frac{2}{d}}+\frac{2 |\partial\Omega'|}{|\Omega'|}\left(\frac{2}{d+2}C_d\right)^{\frac{1}{2}}\kappa^{\frac{1}{d}}+\frac{4|\partial\Omega'|^2}{|\Omega'|^2}.
\end{multline}
This bound shows in particular that semiclassical upper bounds of the form
\begin{equation*}
  \frac{1}{k}\sum_{j=1}^{k}\lambda_j\leq Ak^{2/d}+Bk^{1/d}
\end{equation*}
cannot hold for all $k\geq 1$. Hence an additive constant in \eqref{evsums-DirichletLaplacian15_v30-NEW} is necessary. In the case of a rectangle $\Omega=(0,1)\times(0,L)$, inequality \eqref{ineq-8} reads $\pi^2\leq 16$.

\end{rem}

It is worth mentioning here a paper by Larson \cite{larson}, where the author considers  improved upper estimates for Riesz means $R_{\sigma}$ with $\sigma\geq 3/2$ for convex domains, containing a negative second term which depends only on the measure of the boundary of the domain.
For a certain range of $k$ such bounds
also imply improvements of the Li-Yau lower bounds. We remark that Theorem \ref{thm_convex_first_eigenvalues} reverses the Li-Yau inequalities and presents a positive correction term which, again, depends only on the measure of the boundary, hence complementing the results in \cite{larson} (see also \cite{frank3}).




\subsection{Bounds for various choices of \texorpdfstring{$\phi$}{phi} approximating the characteristic function of regular domains}\label{sec:1.5.2}

The formula \eqref{evsums-DirichletLaplacian1} with $\phi=1_{\Omega}$
would be a ``reverse Berezin-Li-Yau inequality.''  Clearly such an inequality does not hold and
correspondingly we cannot use $\phi\equiv 1$ in \eqref{evsums-DirichletLaplacian1}. Yet the form of inequality \eqref{evsums-DirichletLaplacian1} suggests choosing $\phi$ to be some function in $H^1_0(\Omega)\cap L^{\infty}(\Omega)$ which approximates the constant function $1$.

In this subsection (as well as in Subsections \ref{sec:1.5.3} and \ref{sec:1.5.4}) we present explicit estimates for averages of eigenvalues obtained by suitable choices of cut-off functions $\phi$ in \eqref{evsums-DirichletLaplacian1} for domains satisfying certain regularity properties. Analogous results clearly apply to 
Riesz means by the use of the same cut-off functions in \eqref{Riesz-mean-ineq-DirichletLaplacian}, but for the sake of brevity we shall omit such details.

In order to formulate the main results of this subsection we introduce a few preliminaries.

For $x\in\mathbb R^d$ we denote by $\delta(x)$ the function $\delta(x):=\text{dist}(x,\partial\Omega)$. Let $h>0$ and let $\omega_h\subset\Omega$ be defined by
\begin{equation}\label{tubular_n}
\omega_h:=\left\{x\in\Omega:\delta(x)\leq h\right\}.
\end{equation}
We note that if $h\geq r_{\Omega}$ then $\omega_h=\Omega$ (here $\rho_{\Omega}$ is the inradius of $\Omega$, see \eqref{inradius}).

A suitable function $\phi\in H^1_0(\Omega)$ can be defined by setting $\phi\equiv 1$ on $\Omega\setminus\overline\omega_h$ and then extended to a Lipschitz continuous function in $\Omega$ such that $\phi_{|_{\partial\Omega}}=0$ and $\|\phi\|_{\infty}\leq 1$ in $\Omega$. Then one expects that $|\nabla\phi|\leq 1/h$ a.e. in $\Omega$ and that $\|\phi\|_2^2\sim |\Omega|$, $\|\nabla\phi\|_2^2\sim |\omega_h|/h^2$ as $h\rightarrow 0$. In view of Remark \ref{rem_test} one would then take $h\sim k^{-1/d}$ (for $k$ sufficiently large) in order to obtain asymptotically sharp estimates. It now clear that we need information on the rate at which $|\omega_h|$ goes to zero with respect to $h$, and this is strictly related to the regularity of the boundary.

We introduce the class of domains $\mathcal S$ defined as follows:
\begin{equation*}
\mathcal S:=\left\{\Omega\subset\mathbb R^d{\rm\ bounded\ }:\lim_{h\rightarrow 0^+}\frac{|\omega_h|}{h}=|\partial\Omega|\right\}.
\end{equation*}
The class $\mathcal S$ is related to the notion of {\it outer Minkowski content} (refer to Definition \ref{defn} for more details).
Clearly, if the boundary $\partial\Omega$ is sufficiently smooth, then $\Omega$ belongs to $\mathcal S$. This is the case, for example, of Lipschitz domains (see Proposition \ref{lipschitzprop}).

We can state now the following theorem.

\begin{theorem}\label{main_thm_evsums-DirichletLaplacian}
Let $\Omega\in\mathcal S$. For $k\geq |\Omega|r_{\Omega}^{-d}\left(\frac{d+2}{2C_d}\right)^{d/2}$.  Then
\begin{equation}\label{evsums-DirichletLaplacian14}
\frac1{k}\sum_{j=1}^k\lambda_{j}
\leq \frac{d}{d+2}\,C_d\left(\frac{k}{|\Omega|}\right)^{\frac{2}{d}}\\+2\left(\frac{2 C_d}{d+2}\right)^{\frac{1}{2}}\frac{|\partial\Omega|}{|\Omega|}\left(\frac{k}{|\Omega|}\right)^{\frac{1}{d}}+R(k),
\end{equation}
where $R(k)=o(k^{\frac{1}{d}})$ and depends explicitly on $k$, $d$, $|\Omega|$, $|\partial\Omega|$ and $|\omega_{h(k)}|$, with $h(k)=(d+2)^{1/2}(2C_d)^{-1/2}|\Omega|^{1/d}k^{-1/d}$.
\end{theorem}

\begin{rem}
If $k<|\Omega|r_{\Omega}^{-d}\left(\frac{d+2}{2C_d}\right)^{d/2}$ the upper bound \eqref{evsums-DirichletLaplacian15_v2} clearly still holds.  (See also \eqref{simple_heat_bound_2}.)
\end{rem}

\begin{rem}
It follows from the proof of Theorem \ref{main_thm_evsums-DirichletLaplacian} that if $\Omega$ is such that $|\omega_h|-h|\partial\Omega|=O(h^2)$ as $h\rightarrow 0^+$, then the remainder $R(k)$ in \eqref{evsums-DirichletLaplacian14} satisfies $|R(k)|\leq C$ for some positive constant $C$. We shall discuss examples where this is in fact the case in the next subsection.
\end{rem}

\begin{rem}
We note that the second term in the upper bound \eqref{evsums-DirichletLaplacian14} coincides with the second term of the semiclassical asymptotic expression for the average of Dirichlet eigenvalues \eqref{weyl_D}, up to a multiplicative dimensional constant. In particular one can easily check that
$$
\frac{3}{\sqrt{2}}\leq\frac{2\left(\frac{2C_d}{d+2}\right)^{\frac{1}{2}}}{\left(\frac{1}{2(d+1)}\frac{C_d^{\frac{d+1}{2}}}{C_ {d-1}^{\frac{d-1}{2}}}\right)}\leq\frac{4}{\sqrt{\pi}}
$$
and that the right side of the inequality is the limit of the ratio as $d\rightarrow +\infty$.
\end{rem}

We next consider the partition function. We have the following corollary.

\begin{corollary}\label{thm_bounds_part_func_regular}
Let $\Omega\in\mathcal S$. Then for all $0<t\leq r_{\Omega}^2$
\begin{equation*}
\sum_{j=1}^{\infty}e^{-\lambda_jt}\geq \frac{|\Omega|}{(4\pi t)^{\frac{d}{2}}}-\frac{|\partial\Omega|t^{\frac{1}{2}}}{(4\pi t)^{\frac{d}{2}}}+R(t),
\end{equation*}
where 
\begin{equation*}
R(t):=2\frac{t^{\frac{1}{2}}|\partial\Omega|-|\omega_{t^{\frac{1}{2}}}|}{(4\pi t)^{\frac{d}{2}}}
\end{equation*}
and $R(t)=o(t^{\frac{1-d}{2}})$ as $t\rightarrow 0^+$.
\end{corollary}




\subsection{Estimates for domains of class \texorpdfstring{$C^2$}{C2}}\label{sec:1.5.3}

In the case of domains of class $C^2$ we can provide bounds that are more explicit than those contained in Theorem \ref{main_thm_evsums-DirichletLaplacian}. In fact, in the case of $C^2$ domains (which belong to the class $\mathcal S$), for sufficiently small $h$ it is possible to write a closed formula for the volume of the tube $|\omega_h|$, which has the form of a polynomial in $h$ whose coefficients are given by suitable integrals of the principal curvatures of the boundary. This translates into a more explicit formula for the remainder $R(k)$ in \eqref{main_thm_evsums-DirichletLaplacian} (which will now depend on integrals of the principal curvatures), which further implies a uniform estimate of $R(k)$ with respect to $k$.

We need to introduce some notation first. Let $\bar h$ be defined by
\begin{equation}\label{barh}
\bar h:=\sup\left\{h>0:{\rm every\  point\  in\ } \omega_{h} {\rm\ has\ a\ unique\ nearest\ point\ on\ } \partial\Omega\right\}.
\end{equation}
It is well-known (see also Theorem \ref{tubular0}) that if $\Omega$ is a domain of class $C^2$, such
an $\bar h$ exists and is strictly positive.

We are ready to state the main theorem of this subsection.
\begin{theorem}\label{smooth_1}
Let $\Omega$ be a bounded domain of class $C^2$ in $\mathbb R^d$. Then Theorem \ref{main_thm_evsums-DirichletLaplacian} holds. Moreover:
\begin{itemize}
\item[(i)] There exists $C>0$ which depends only on $\Omega$ and $d$ such that \eqref{evsums-DirichletLaplacian14} holds with $|R(k)|\leq C$.
\item[(ii)] For $k\geq|\Omega|\bar h^{-d}\left(\frac{d+2}{2C_d}\right)^{\frac{d}{2}}$, we have the following explicit formula for the remainder in \eqref{evsums-DirichletLaplacian14}:
\small
\begin{multline}\label{remainder-f1}
R(k)\\
=\frac{2|\partial\Omega|^2+2\left(h(k)|\partial\Omega|+|\Omega|\right)\frac{1}{d}\sum_{j=2}^d\binom{d}{j}(-1)^{j-1}h(k)^{j-2}\int_{\partial\Omega}\mathcal H(x)^{j-1}d\sigma(x)}{|\Omega|^2-h(k)|\Omega||\partial\Omega|-\frac{|\Omega|}{d}\sum_{j=2}^d\binom{d}{j}(-1)^{j-1}h(k)^j\int_{\partial\Omega}\mathcal H(x)^{j-1}d\sigma(x)},\\
\end{multline}
\normalsize
where $\mathcal H(x)$ denotes the mean curvature of $\partial\Omega$ at a point $x$ and $h(k)$ is defined by 
$$
h(k):=\left(\frac{2}{d+2}C_d\right)^{-\frac{1}{2}}\left(\frac{k}{|\Omega|}\right)^{-\frac{1}{d}}.
$$
In particular
\begin{equation}\label{limit_r1}
\lim_{k\rightarrow +\infty}R(k)=2\frac{|\partial\Omega|^2}{|\Omega|^2}-\frac{d-1}{|\Omega|}\int_{\partial\Omega}\mathcal H(x)d\sigma(x).
\end{equation}
\end{itemize} 
\end{theorem}

Estimates can be improved if some additional properties are satisfied. We say that a bounded domain of class $C^2$ in $\mathbb R^d$ is {\em mean convex} if $\mathcal H(s)\geq 0$ for all $s\in\partial\Omega$. We have the following corollary.
\begin{corollary}\label{smooth_convex}
Let $\Omega$ be a bounded domain of class $C^2$ in $\mathbb R^d$ which is mean convex. Then Theorem \ref{main_thm_evsums-DirichletLaplacian} holds. Moreover:
\begin{itemize}
\item[(i)] There exists $C>0$ which depends only on $\Omega$ and $d$ such that \eqref{evsums-DirichletLaplacian14} holds with $|R(k)|\leq C$.
\item[(ii)] For all  $k\geq|\Omega|\left(\frac{d+2}{2C_d}\cdot\frac{|\Omega|-2\bar h|\partial\Omega|}{\bar h^2|\Omega|}\right)^{\frac{d}{2}}$,
\begin{equation}\label{evsums-DirichletLaplacian15_v30-NEW-2}
  \frac{1}{k}\sum_{j=1}^{k}\lambda_j\leq \frac{d}{d+2}\,C_d\left(\frac{k}{|\Omega|}\right)^{\frac{2}{d}}+2\left(\frac{2}{d+2}\,C_d\right)^{\frac{1}{2}}\left(\frac{k}{|\Omega|}\right)^{\frac{1}{d}}\frac{|\partial \Omega|}{|\Omega|}+\frac{4|\partial \Omega|^2}{|\Omega|^2}.
	\end{equation}
\item[(iii)] If $\bar h\geq\frac{|\Omega|}{2|\partial\Omega|}$, then inequality \eqref{evsums-DirichletLaplacian15_v30-NEW-2} holds for all positive integers $k$ .

\end{itemize}
\end{corollary}

\begin{rem}
We remark that in order to have a second term with the right power of $k$ and a remainder $R(k)$ uniformly bounded in $k$ in Theorem \ref{main_thm_evsums-DirichletLaplacian}, much less regularity than $C^2$ is required. In fact, if the set $\mathbb R^d\setminus\Omega$ has positive reach (see e.g., \cite{federer} for the precise definition), then
$$
|\omega_h|=\sum_{i=0}^{d-1}h^{d-i}\Phi_i(\Omega),
$$
where the coefficients $\Phi_i(\Omega)$ depend only on the domain $\Omega$ and are the so-called curvature measures of $\Omega$ (up to dimensional constants). We refer to \cite{federer} for more information on curvature measures and sets with positive reach. Hence if $\mathbb R^d\setminus\Omega$ is a set with positive reach, then point (i) of Theorem \ref{smooth_1} remains valid, up to
possibly substituting
for $|\partial\Omega|$ in the second summand in the left side of \eqref{evsums-DirichletLaplacian14} a suitable quantity which depends only on $\Omega$. Moreover, point (i) of Theorem \ref{smooth_1} holds also when $\Omega$ is a $d$-dimensional polyhedron. In fact in this case the volume of the tube $|\omega_h|$ about $\Omega$ when $h$ is sufficiently small is given by a polynomial of degree $d$ in the variable $h$ and the lower order term is given by $h|\partial\Omega|$ (see e.g., \cite{cheegertubes}).
\end{rem}




\subsection{Estimates for planar sets}\label{sec:1.5.4}
In the case of planar sets we are able to provide bounds which depend on relevant features of the domain. Namely, if the planar domain is of class $C^2$, the estimates depend explicitly only on $k,|\Omega|, |\partial\Omega|$ and on the number of connected components of the boundary. Estimates are improved when the boundary has at most two connected components or when it is convex. In the case of a polygon, we show that bounds depend explicitly only on $k,|\Omega|, |\partial\Omega|$ and on the value of its angles. In this case we note that bounds depend on fairly simple geometric quantities which are easily computable, despite the constant of the second term is not the correct one (but differs for a dimensional factor). On the other hand, Weyl's law tells us that asymptotically sharp upper bounds hold, however we are in general not able to compute them explicitly.

We have the following theorem.

\begin{theorem}\label{planar_thm}
Let $\Omega$ be a bounded domain in $\mathbb R^2$. Then the following statements hold:
\begin{itemize}
\item[(i)] If $\Omega$ is of class $C^2$ then for all $k\geq\frac{|\Omega|}{2\pi\bar h^2}$
\begin{equation}\label{planarC2}
\frac1{k}\sum_{j=1}^k\lambda_{j}
\leq 2\pi\frac{k}{|\Omega|}\\+\left(8\pi\right)^{\frac{1}{2}}\frac{|\partial\Omega|}{|\Omega|}\left(\frac{k}{|\Omega|}\right)^{\frac{1}{2}}+R(k),
\end{equation}
where
\begin{equation}\label{remainder_planarC2}
R(k)=\frac{2\left(|\partial\Omega|^2-(2-b)\pi|\Omega|-(2-b)\pi |\partial\Omega| h(k)\right)}{|\Omega|\left(|\Omega|-|\partial\Omega|h(k)+(2-b)\pi h(k)^2\right)},
\end{equation}
with $h(k)=(|\Omega|/(2k\pi))^{1/2}$ and $b$ denoting the number of connected components of $\partial\Omega$. In particular,
$$
\lim_{k\rightarrow+\infty} R(k)=\frac{2(|\partial\Omega|^2-(2-b)\pi |\Omega|)}{|\Omega|^2}.
$$
\item[(ii)] If $\Omega$ is of class $C^2$ and $\partial\Omega$ has at most $2$ connected components, then for all $0<\alpha<1$ and $k\geq\frac{|\Omega|}{2\pi}\max\left\{\bar h^{-2},\left(\frac{|\partial\Omega|}{\alpha|\Omega|}\right)^2\right\}$,
\begin{equation*}
\frac1{k}\sum_{j=1}^k\lambda_{j}
\leq 2\pi\frac{k}{|\Omega|}\\+\left(8\pi\right)^{\frac{1}{2}}\frac{|\partial\Omega|}{|\Omega|}\left(\frac{k}{|\Omega|}\right)^{\frac{1}{2}}+\frac{2|\partial\Omega|^2}{(1-\alpha)|\Omega|^2}.
\end{equation*}
\item[(iii)] If $\Omega$ is convex, then for all positive integer $k$,
\begin{equation*}
\frac1{k}\sum_{j=1}^k\lambda_{j}
\leq 2\pi\frac{k}{|\Omega|}\\+\left(8\pi\right)^{\frac{1}{2}}\frac{|\partial\Omega|}{|\Omega|}\left(\frac{k}{|\Omega|}\right)^{\frac{1}{2}}+\frac{4|\partial\Omega|^2}{|\Omega|^2}.
\end{equation*}
\item[(iv)] If $\Omega$ is a polygon with perimeter given by $|\partial\Omega|$, $n_a\in\mathbb N$ angles $\left\{\alpha_i\right\}_{i=1}^{n_a}$ with $0<\alpha_i<\pi$, and $n_b\in\mathbb N$ angles $\left\{\beta_i\right\}_{i=1}^{n_b}$ with $\pi<\beta_i<2\pi$, then there exists $\tilde h>0$ such that \eqref{planarC2} holds for $k\geq\frac{|\Omega|}{2\pi\tilde h^2}$, with
\begin{equation}\label{remainder_polygon}
R(k)=\frac{2\left(|\partial\Omega|^2-(S_A-S_B)|\Omega|-(S_A-S_B) |\partial\Omega| h(k)\right)}{|\Omega|\left(|\Omega|-|\partial\Omega|h(k)+(S_A-S_B)h(k)^2\right)},
\end{equation}
where
$$
S_A=\sum_{i=1}^{n_a}\cot(\alpha_i/2)
$$
and
$$
S_B=\sum_{i=1}^{n_b}\frac{(\beta_i-\pi)}{2}.
$$
In particular,
$$
\lim_{k\rightarrow+\infty} R(k)=\frac{2(|\partial\Omega|^2-(S_A-S_B)|\Omega|)}{|\Omega|^2}.
$$
\end{itemize}
\end{theorem}




\section{The main results for Neumann eigenvalues}\label{sec:1.6}

In this section we discuss refined upper bounds and new lower bounds for the eigenvalues of $-\Delta_{\Omega}^N$, the self-adjoint realization of the (nonnegative) Laplacian on $\Omega$ with Neumann boundary conditions. Throughout this section we shall always assume that $\Omega$ is a bounded domain in $\mathbb R^d$ such that the spectrum of $-\Delta_{\Omega}^N$ is discrete.   (This is true, e.g., if the embedding $H^1(\Omega)\subset L^2(\Omega)$ is compact.) Here and in what follows we denote by $\left\{\mu_j\right\}_{j=1}^{\infty}$ the set of (nonnegative) eigenvalues of $-\Delta_{\Omega}^N$, and by $\left\{v_j\right\}_{j=1}^{\infty}$ the corresponding orthonormal set in $L^2(\Omega)$ of eigenfunctions.

Weyl-type upper bounds for sums of Neumann eigenvalues \eqref{kroger_0} are known from \cite{Kro};
however such bounds can be improved by means of the averaged variational principle (see also \cite{harrell_stubbe_18}).  We obtain improvements of the classical Kr\"oger inequality \eqref{kroger_0} as a corollary of Theorem \ref{dirichlet_laplacian_thm_general}.




\begin{theorem}\label{refinement_kroger_thm}
For all positive integers $k$,
\begin{equation}\label{Neumann-Laplacian-ev-bound-normalized2}
  \frac{d+2}{d}\frac{1}{k}\sum_{j=1}^{k}\mu_j-C_d\left(\frac{k}{|\Omega|}\right)^{\frac{2}{d}}\leq -C_d\left(\frac{k}{|\Omega|}\right)^{\frac{2}{d}}\,\left(C_d\left(\frac{k}{|\Omega|}\right)^{-\frac{2}{d}}\mu_{k+1}-1\right)^2,
\end{equation}
and for single eigenvalues,
\begin{equation}\label{Neumann-Laplacian-ev-bound-normalized3}
 C_d\left(\frac{k}{|\Omega|}\right)^{\frac{2}{d}}x_-\leq \mu_k\leq\mu_{k+1} \leq C_d\left(\frac{k}{|\Omega|}\right)^{\frac{2}{d}}x_+,
\end{equation}
where
\begin{equation*}
x_{\pm}=1\pm\sqrt{1-\frac{\frac{d+2}{d}\frac{1}{k}\sum_{j=1}^{k}\mu_j}{C_d\left(\frac{k}{|\Omega|}\right)^{\frac{2}{d}}}}.
\end{equation*}
For all $z>0$
\begin{equation}\label{riesz_means_neumann}
  \sum_{j}(z-\mu_{j})_{+}\geq \frac{2}{d+2}C_d^{-\frac{d}{2}}\,|\Omega|z^{1+d/2}.
\end{equation}
For the partition function
\begin{equation}\label{part_func_neumann}
\sum_{j=1}^{\infty}e^{-\mu_j t}\geq\frac{|\Omega|}{(4\pi t)^{\frac{d}{2}}},
\end{equation}
\end{theorem}

We remark that \eqref{riesz_means_neumann} and \eqref{part_func_neumann} are known, classical results (see \cite{Lap1997}), which we obtain as by-products Theorem \ref{refinement_kroger_thm}.

For completeness we recall the two-term lower bound for the Riesz mean $R_1(z)$ proved in
\cite{harrell_stubbe_18} by means of the averaged variational principle:
\begin{theorem}
For each unit vector $v\in\mathbb R^d$ and for all $z\geq 0$
\begin{multline*}
\sum_j(z-\mu_j)_+\\
\geq \frac{2}{d+2}C_d^{-\frac{d}{2}}|\Omega|z^{\frac{d}{2}+1}
+
\left(\frac{1}{4}\frac{2}{d+1}C_{d-1}^{-\frac{d-1}{2}}\frac{|\Omega|}{\delta_v(\Omega)}z^{\frac{d}{2}+\frac{1}{2}}
-\frac{1}{96}(2\pi)^{2-d}\omega_d\frac{|\Omega|}{\delta_v(\Omega)^2}z^{\frac{d}{2}}\right)_+,
\end{multline*}
where
$$
\delta_v(\Omega):=\sup\left\{v\cdot(x-y):x,y\in\Omega\right\}.
$$
\end{theorem}




As another application of the averaged variational principle we obtain Weyl-type upper bounds for Riesz means and lower bounds for averages, which are sharp in the semiclassical limit.

\begin{theorem}\label{thm_bounds_neumann_eigenvalues}
Let $\Omega$ be a bounded domain in $\mathbb R^d$ of class $C^2$.\\
For all $z\geq \max\left\{\bar h^{-2},\frac{4}{9}\max_{\substack{x\in\partial\Omega\\0\leq h\leq\bar h/2}}\left|\sum_{i=1}^{d-1}\frac{h\kappa_i(x)}{1-h\kappa_i(x)}\right|^{2}\right\}$ the following inequality holds.
\begin{equation}\label{riesz-neu-5}
\sum_{j}(z-\mu_j)_+\leq \frac{2}{d+2}C_d^{-\frac{d}{2}}|\Omega|z^{1+\frac{d}{2}}+\pi|\partial\Omega|c_dz^{\frac{d}{2}+\frac{1}{2}}+R'(z),
\end{equation}
where $c_d$ is a constant depending only on $d$ and $R'(z)$ depends explicitly on $z$, $d$, $|\partial\Omega|$ and $|\omega_{h(z)}|$ with $h(z)=\frac{\pi}{2\sqrt{z}}$. Moreover,
$$
\lim_{z\rightarrow+\infty}\frac{|R'(z)|}{z^{d/2}}\leq\frac{\pi^2(d-1)c_d}{4}\left|\int_{\partial\Omega}\mathcal H(x)d\sigma(x)\right|.
$$
For all $k\geq c_d|\Omega|\max\left\{\bar h^{-d},(2/3)\max_{\substack{x\in\partial\Omega\\0\leq h\leq\bar h/2}}\left|\sum_{i=1}^{d-1}\frac{h\kappa_i(x)}{1-h\kappa_i(x)}\right|^{d}\right\}$ the following inequality holds.
\begin{equation}\label{bounds_neumann_eigenvalues}
\frac{1}{k}\sum_{j=1}^k\mu_j\geq\frac{d}{d+2}C_d\left(\frac{k}{|\Omega|}\right)^{\frac{2}{d}}-\pi c_dC_d^{\frac{d+1}{2}}\frac{|\partial\Omega|}{|\Omega|}\left(\frac{k}{|\Omega|}\right)^{\frac{1}{d}}-R(k),
\end{equation}
where $R(k)$ depends explicitly on $k$, $d$, $|\Omega|$, $|\partial\Omega|$ and $|\omega_{h(k)}|$ with \\$h(k)=\frac{\pi}{2}C_d^{-1/2}|\Omega|^{1/d}k^{-1/d}$. Moreover,
$$
\lim_{k\rightarrow+\infty}|R(k)|\leq\frac{\pi^2(d-1)c_dC_d^{\frac{d}{2}}}{4|\Omega|}\left|\int_{\partial\Omega}\mathcal H(x)d\sigma(x)\right|.
$$
\end{theorem}

\begin{rem}
We note that the second term in the lower bound \eqref{bounds_neumann_eigenvalues} coincides with the second term of the asymptotic expansion expression of the sum of Neumann eigenvalues \eqref{weyl_N}, up to a multiplicative dimensional constant.
\end{rem}




\section{Application of the AVP to the Dirichlet\texorpdfstring{\\}{}Laplacian: proofs of the main results}\label{sec:2}

This section collects all the proofs of the Theorems presented in Section \ref{sec:1.5} as well as the proofs of the corresponding corollaries. Such proofs consist in the application of the averaged variational principle \eqref{averaged-var-principle-Rieszmean2_0}-\eqref{averaged-var-principle-Rieszmean2} with suitable families of test functions. 

Here and in the sequel, for a function $f\in L^1(\mathbb R^d)$ we denote by $\hat f(\xi)$ its Fourier transform  defined by $\hat f(\xi):=(2\pi)^{-d/2}\int_{\mathbb R^d}f(x)e^{i\xi\cdot x}dx$, and with abuse of notation, for a function $f\in H^1_0(\Omega)$ we still denote by $\hat f(\xi)$ the Fourier transform of its extension by zero to $\mathbb R^d$.

\begin{proof}[Proof of Theorem \ref{dirichlet_laplacian_thm_general}]
We take in \eqref{averaged-var-principle-Rieszmean2_0} trial functions of the form \\
$\displaystyle f_{\xi}(x)=(2\pi)^{-d/2}e^{i\xi x}\phi(x)$ with $\phi\in H^1_0(\Omega)\cap L^{\infty}(\Omega)$. After averaging over $\xi\in \mathbb{R}^d$ and using the unitarity of the Fourier transform, for any weight $w(\xi), 0\leq w(\xi)\leq 1$
 we get:
\begin{multline}\label{avp-DirichletLplacian1}
  \sum_{j=1}^k(\lambda_{k+1}-\lambda_{j})\int_{\Omega}\phi^2(x)u_j(x)^2\;dx\\
	\geq (2\pi)^{-d}\int_{\mathbb{R}^d}\left((\lambda_{k+1}-|\xi|^2)||\phi||_2^2-||\nabla\phi||_2^2\right)w(\xi)\,d\xi.
\end{multline}
Choosing $\displaystyle w(\xi)=1_{\{\xi\in\mathbb R^d:|\xi|\leq R\}}$, inequality \eqref{avp-DirichletLplacian1} immediately implies that
\begin{multline}\label{general_AVP-DirichletLaplacian}
  \sum_{j=1}^k(\lambda_{k+1}-\lambda_{j})\int_{\Omega}\phi^2(x)u_j(x)^2\;dx\\
  \geq (2\pi)^{-d}\omega_d||\phi||_2^2\left(\left(\lambda_{k+1}-\frac{||\nabla\phi||_2^2}{||\phi||_2^2}\right)R^d-\frac{d}{d+2}R^{d+2}\right).\\
 \end{multline}
We note that inequality \eqref{general_AVP-DirichletLaplacian} holds with $\lambda_{k+1}$ replaced by any $z\in[\lambda_k,\lambda_{k+1}]$, and hence we have (see also \eqref{averaged-var-principle-Rieszmean2})
\begin{multline}\label{general_AVP-DirichletLaplacian01}
  \sum_j(z-\lambda_{j})_+\int_{\Omega}\phi^2(x)u_j(x)^2\;dx\\
  \geq (2\pi)^{-d}\omega_d||\phi||_2^2\left(\left(z-\frac{||\nabla\phi||_2^2}{||\phi||_2^2}\right)R^d-\frac{d}{d+2}R^{d+2}\right).\\
 \end{multline}
Now, by taking
$$
R^2=\left(z-\frac{||\nabla\phi||_2^2}{||\phi||_2^2}\right)_{+}
$$
in \eqref{general_AVP-DirichletLaplacian01} we get
\begin{equation}\label{avp-DirichletLplacian2}
 \sum_j(z-\lambda_{j})_+\int_{\Omega}\phi^2(x)u_j(x)^2\;dx
	\geq \frac{2}{d+2}\,(2\pi)^{-d}\omega_d||\phi||_2^2\left(z-\frac{||\nabla\phi||_2^2}{||\phi||_2^2}\right)_{+}^{\frac{d}{2}+1}.
\end{equation}
This proves \eqref{Riesz-mean-ineq-DirichletLaplacian}.

Inequality \eqref{evsums-DirichletLaplacian1} follows from \eqref{general_AVP-DirichletLaplacian} by taking
$$
R^2=C_d\left(\frac{k}{|\Omega|}\right)^{\frac{2}{d}}\rho(\phi)^{-2/d}
$$
and from the fact that $\int_{\Omega}\phi^2 u_j^2 dx\leq\|\phi\|_{\infty}^2$. This concludes the proof.
\end{proof}

\begin{proof}[Proof of Corollary \ref{dirichlet_bounds_spec_fcn}]
Laplace transforming \eqref{Riesz-mean-ineq-DirichletLaplacian} immediately yields inequality \eqref{Partition-function-inequality}, which implies that
\begin{equation}\label{part-fct-estimate}
  \sum_{j=1}^{\infty}e^{-\lambda_{j}t}\geq \frac{|\Omega|\rho(\phi)}{(4\pi t)^{\frac{d}{2}}}\cdot e^{-\frac{||\nabla\phi||_2^2}{||\phi||_2^2}\,t}
\end{equation}
for all $t>0$. In view of the semiclassical expansion we are interested in bounds for small $t$, and therefore we apply the inequality $\displaystyle e^{-x}\geq 1-x$ to \eqref{part-fct-estimate}, from which we get
\begin{equation*}
  \sum_{j=1}^{\infty}e^{-\lambda_{j}t}\geq \frac{|\Omega|\rho(\phi)}{(4\pi t)^{\frac{d}{2}}}\cdot \left(1 -\frac{||\nabla\phi||_2^2}{||\phi||_2^2}\,t	\right),
\end{equation*}
immediately implying \eqref{part-fct-estimate-small-times}. This concludes the proof.
\end{proof}

\begin{proof}[Proof of Corollary \ref{quadratic-eigenvalue-inequality-general}]
The proof follows that of \cite[Theorem 1.1]{harrell_stubbe_18}, replacing the $\mu_j$ by $\displaystyle \lambda_{j}-\frac{||\nabla\phi||_2^2}{||\phi||_2^2}$ and $m_k$ by $m_k\rho(\phi)^{-d/2}$, and using inequality \eqref{general_AVP-DirichletLaplacian}.
\end{proof}




\subsection{Upper bounds with no restriction on the regularity of the boundary: proofs}\label{sub:2:1}

\begin{proof}[Proof of Theorem \ref{simple_bound}]

We choose $\phi=u_1$ in Theorem \ref{dirichlet_laplacian_thm_general}, where $u_1$ is the eigenfunction associated with the first Dirichlet eigenvalue on $\Omega$. We recall the optimal upper bound (see \cite{vdB})
\begin{equation}\label{u1-upperbound}
  |u_1(x)|^2\leq 2d(2\pi)^{-d}\omega_d \frac{\lambda_1(B)^{\frac{d-1}{2}}}{|J_{\frac{d}{2}}(\sqrt{\lambda_1(B)})|}\,r_{\Omega}^{-d},
\end{equation}
where $B$ denotes the unit ball in $\mathbb R^d$, $J_{\nu}$ denotes the Bessel function of the first kind and order $\nu$, and $r_{\Omega}$ is the inradius of $\Omega$ (see \eqref{inradius}). The bound \eqref{u1-upperbound} is saturated when $\Omega$ is a ball . We note that inequality \eqref{u1-upperbound} holds for bounded domains in $\mathbb R^d$ with no further regularity assumptions on the boundary. By using \eqref{u1-upperbound} in \eqref{Riesz-mean-ineq-DirichletLaplacian} and \eqref{evsums-DirichletLaplacian1} we obtain \eqref{avp-DirichletLplacian3} and \eqref{evsums-DirichletLaplacian15_v2}. This concludes the proof.
\end{proof}

\begin{rem}\label{remark1}
We note that bound  \eqref{evsums-DirichletLaplacian15_v2} depends on $\lambda_1$. In order to have a bound which depends only on $k,|\Omega|$, and $r_{\Omega}$, we need an upper bound on $\lambda_1$. We refer, e.g., to \cite{geometrical_structure} for a review of geometric inequalities for eigenvalues. A simple upper bound is, for example, the following: 
\begin{equation}\label{up_bd_lambda1_1}
\lambda_1\leq\frac{\lambda_1(B)}{r_{\Omega}^2}.
\end{equation}
In fact, from the variational principle for $\lambda_1$,
\begin{equation}\label{up_bd_lambda1}
\lambda_1\leq\frac{\|\nabla\phi\|_2^2}{\|\phi\|_2^2}
\end{equation}
for all $\phi\in H^1_0(\Omega)$.  Taking $\phi=u_{1,r_{\Omega}}$ , where $u_{1,r_{\Omega}}$ is the first Dirichlet eigenfunction on $B_{r_{\Omega}}\subseteq\Omega$  extended by $0$ and $B_{r_{\Omega}}$ is a ball of radius $r_{\Omega}$ contained in $\Omega$, we immediately obtain \eqref{up_bd_lambda1_1}, which now complements bound \eqref{evsums-DirichletLaplacian15_v2}.
\end{rem}

\begin{rem}\label{remark2}
We remark that a simpler upper bound for $u_1$ is given by the standard heat kernel estimate (see e.g., \cite{davies_heat}):
\begin{equation*}
  |u_1(x)|^2\leq \bigg(\frac{e\lambda_1}{2d\pi}\bigg)^{\frac{d}{2}},
\end{equation*}
which yields the more explicit bounds
\begin{equation*}
  \sum_{j}\frac{(z-\lambda_{j})_+}{\lambda_{1}}\geq \bigg(\frac{d}{2e}\bigg)^{\frac{d}{2}}\frac{1}{\Gamma(\frac{d}{2}+2)}\left(\frac{z}{\lambda_{1}}-1\right)_+^{\frac{d}{2}+1}
\end{equation*}
for all $z>0$, and
\begin{equation}\label{simple_heat_bound_2}
	\frac{1}{k}\sum_{j=1}^k(\lambda_{j}-\lambda_1)\leq \frac{d}{d+2}C_dk^{\frac{2}{d}}\left(\frac{e\lambda_1}{2d\pi}\right)
\end{equation}
for all positive integers $k$.
\end{rem}

In order to prove Theorem \ref{thm_convex_first_eigenvalues} we introduce some preliminaries.

We define a function $\phi_h$ such that $\phi_h\equiv 1$ in $\Omega\setminus\overline\omega_h$, 
 $0\leq\phi_h\leq 1$ in $\omega_h$, ${\phi_h}_{|_{\partial\Omega}}=0$ in the following way. Let $f:[0,1]\rightarrow [0,1]$ be a continuously differentiable function such that $f(0)=0$ and $f(1)=1$. Then we set
\begin{equation}\label{phi-h-test-fct}
  \phi_h(x)=\left\{
            \begin{array}{ll}
              1, & \hbox{if $x\in\Omega\setminus\overline\omega_h$;} \\
              f(\frac{\delta(x)}{h}), & \hbox{if $x\in\omega_h$.}
            \end{array}
          \right.
\end{equation}
We note that for all $h>r_{\Omega}$, $\|\phi_h\|_{\infty}=r_{\Omega}/h<1$, and $\|\phi_h\|_{\infty}\rightarrow 0$ as $h\rightarrow+\infty$. Clearly $\phi_h(x)\in H^1_0(\Omega)$ without further regularity assumptions on $\Omega$. 

Choosing, for example, $\displaystyle f(p)=p$, we have $||\phi_h||_{\infty}=1$,
\begin{equation}\label{up_2_phi_h}
||\phi_h||_2^2=|\Omega|-|\omega_h|+\int_{\omega_h}\phi_h^2(x)dx\geq |\Omega|-|\omega_h|
\end{equation}
and
\begin{equation}\label{up_2_nabla_phi_h}
||\nabla\phi_h||_2^2=\frac{|\omega_h|}{h^2},
\end{equation}
where this last fact is a consequence of the Lipschitz continuity of the distance function $\delta(x)$ on $\mathbb R^d$ and of the fact that $|\nabla\delta(x)|=1$ for almost all $x\in\mathbb R^d$.

We are now ready to prove Theorem \ref{thm_convex_first_eigenvalues}.

\begin{proof}[Proof of Theorem \ref{thm_convex_first_eigenvalues}]

We start by proving \eqref{evriesz-DirichletLaplacian15_v30-NEW}. From \eqref{Riesz-mean-ineq-DirichletLaplacian} it follows that for any $\phi\in H^1_0(\Omega)\cap L^{\infty}(\Omega)$,
\begin{equation}\label{ineq-4}
  \sum_{j}(z-\lambda_j)_{+}\geq \frac{2}{d+2}\,(2\pi)^{-d}\omega_d
  \frac{||\phi||_2^2}{||\phi||_{\infty}^2}\left(z-\frac{||\nabla\phi||_2^2}{||\phi||_2^2}\right)_{+}^{\frac{d}{2}+1}.
\end{equation}
Applying Bernoulli's inequality yields
\begin{equation}\label{ineq-5}
  \sum_{j}(z-\lambda_j)_{+}\geq \frac{2}{d+2}\,(2\pi)^{-d}\omega_d
  \frac{||\phi||_2^2}{||\phi||_{\infty}^2}z^{\frac{d}{2}+1}\left(1-\frac{||\nabla\phi||_2^2}{z||\phi||_2^2}\left(1+\frac{d}{2}\right)\right).
\end{equation}
In order to get an explicit estimate we need a suitable choice for $\phi$ and upper bounds on $\displaystyle \frac{||\phi||_{\infty}^2}{||\phi||_2^2}$ and $\displaystyle \frac{||\nabla\phi||_2^2}{||\phi||_2^2}$.
In \eqref{ineq-5} we choose $\phi=\phi_h$, where $\phi_h$ is defined by \eqref{phi-h-test-fct} with $f(p)=p$.

Now we recall that if $\Omega$ is convex, then $|\omega_h|\leq h|\partial\Omega|$  for all $h\leq r_{\Omega}$. This follows from the co-area formula and from the fact that the Hausdorff measure of the sets $\partial\Omega_h=\left\{x\in\Omega:{\rm dist}(x,\partial\Omega)=h\right\}$ is a non-increasing function of $h$ for $h\in[0,r_{\Omega}]$. In the same way one proves that for a convex domain, $r_{\Omega}\geq\frac{|\Omega|}{|\partial\Omega|}$.  With these facts and \eqref{up_2_phi_h} and \eqref{up_2_nabla_phi_h}, from inequality \eqref{ineq-5} we deduce that
\begin{equation*}
  \sum_{j}(z-\lambda_j)_{+}\geq \frac{2}{d+2}\,(2\pi)^{-d}\omega_d \left(|\Omega|-h|\partial\Omega|\right)\left(1-\frac{(d+2)|\partial\Omega|}{2zh(|\Omega|-h|\partial\Omega|)}\right),
\end{equation*}
which can be rewritten as
\begin{equation}\label{ineq-6}
  \sum_{j}(z-\lambda_j)_{+}\geq \frac{2}{d+2}\,(2\pi)^{-d}\omega_d |\Omega|z^{1+d/2}(1-X)\left(1-\frac{d+2}{2z\frac{|\Omega|^2}{|\partial\Omega|^2}X(1-X)}\right),
\end{equation}
where $X=\frac{h|\partial\Omega|}{|\Omega|}$. The optimizing $X$ is given by $X=X_0:=\frac{|\partial\Omega|}{|\Omega|\sqrt{z}}\sqrt{\frac{d+2}{2}}$, which yields the bound
\begin{equation}\label{ineq-7}
  \sum_{j}(z-\lambda_j)_{+}\geq \frac{2}{d+2}\,(2\pi)^{-d}\omega_d |\Omega|z^{\frac{d}{2}+1}-2\sqrt{\frac{2}{d+2}}(2\pi)^{-d}\omega_d |\partial\Omega|z^{\frac{d}{2}+\frac{1}{2}}.
\end{equation}
The choice of $X_0$ is admissible provided that $X_0\leq\frac{r_{\Omega}|\partial\Omega|}{|\Omega|}$, which is equivalent to the condition $z\geq\frac{d+2}{2r_{\Omega}^2}$. This proves \eqref{evriesz-DirichletLaplacian15_v30-NEW}.

Let us now prove \eqref{evsums-DirichletLaplacian15_v30-NEW}. From \eqref{evsums-DirichletLaplacian1} it follows that for any $\phi\in H^1_0(\Omega)\cap L^{\infty}(\Omega)$,
\begin{multline}\label{evsums-DirichletLaplacian12-NEW}
 \frac1{k}\sum_{j=1}^k\lambda_{j}\leq \frac{d}{d+2}\,C_d\left(\frac{k}{|\Omega|}\right)^{\frac{2}{d}}+\frac{d}{d+2}\,C_d\left(\frac{k}{|\Omega|}\right)^{\frac{2}{d}}\left(\rho(\phi)^{-2/d}-1\right)+\frac{||\nabla\phi||_2^2}{||\phi||_2^2}\\
\leq \frac{d}{d+2}\,C_d\left(\frac{k}{|\Omega|}\right)^{\frac{2}{d}}+\frac{2}{d+2}\,C_d\left(\frac{k}{|\Omega|}\right)^{\frac{2}{d}}\left(\rho(\phi)^{-1}-1\right)+\frac{||\nabla\phi||_2^2}{||\phi||_2^2}\\
=\frac{d}{d+2}\,C_d\left(\frac{k}{|\Omega|}\right)^{\frac{2}{d}}+\frac{2}{d+2}\,C_d\left(\frac{k}{|\Omega|}\right)^{\frac{2}{d}}\left(\frac{|\Omega|\cdot||\phi||_{\infty}^2}{||\phi||_2^2}-1\right)+\frac{||\nabla\phi||_2^2}{||\phi||_2^2}.
\end{multline}
We proceed as above and choose $\phi=\phi_h$ in \eqref{evsums-DirichletLaplacian12-NEW}, where $\phi_h$ is defined by \eqref{phi-h-test-fct} with $f(p)=p$. Thanks to \eqref{up_2_phi_h} and \eqref{up_2_nabla_phi_h}, inequality \eqref{evsums-DirichletLaplacian12-NEW} implies that
\begin{equation}\label{evsums-DirichletLaplacian13-NEW}
 \frac1{k}\sum_{j=1}^k\lambda_{j}\leq \frac{d}{d+2}\,C_d\left(\frac{k}{|\Omega|}\right)^{\frac{2}{d}}+\left(\frac{2}{d+2}\,C_d\left(\frac{k}{|\Omega|}\right)^{\frac{2}{d}}+\frac{1}{h^2}\right)\cdot\frac{|\omega_h|}{|\Omega|-|\omega_h|},
\end{equation}
for all $h\leq r_{\Omega}$. We can rewrite \eqref{evsums-DirichletLaplacian13-NEW} as follows:
\begin{equation}\label{evsums-DirichletLaplacian13-NEW-2}
 \frac1{k}\sum_{j=1}^k\lambda_{j}\leq \frac{d}{d+2}\,C_d\left(\frac{k}{|\Omega|}\right)^{\frac{2}{d}}+\frac{2}{d+2}\,C_d\left(\frac{k}{|\Omega|}\right)^{\frac{2}{d}}\cdot\frac{X}{1-X}+\frac{|\partial\Omega|^2}{|\Omega|^2X(1-X)},
\end{equation}
where $X=\frac{h|\partial\Omega|}{|\Omega|}$. The right
side of inequality \eqref{evsums-DirichletLaplacian13-NEW} is optimized when 
\begin{equation}\label{optX0}
X=X_0:=\frac{1}{1+\sqrt{\frac{2}{d+2}C_d\left(\frac{k}{|\Omega|}\right)^{2/d}\frac{|\Omega|^2}{|\partial\Omega|^2}+1}}.
\end{equation}
We note that this choice is always admissible for convex domains since $r_{\Omega}\geq\frac{|\Omega|}{|\partial\Omega|}$. Inserting $X=X_0$ into \eqref{evsums-DirichletLaplacian13-NEW} we obtain
\begin{equation}\label{evsums-DirichletLaplacian13-NEW-3}
  \frac{1}{k}\sum_{j=1}^{k}\lambda_j\leq \frac{d}{d+2}\,C_d\left(\frac{k}{|\Omega|}\right)^{\frac{2}{d}}+\frac{2|\partial\Omega|^2}{|\Omega|^2X_0}.
\end{equation}
Using the trivial inequality 
$$
1+\sqrt{\frac{2}{d+2}C_d\left(\frac{k}{|\Omega|}\right)^{2/d}\frac{|\Omega|^2}{|\partial\Omega|^2}+1}\leq 2+\sqrt{\frac{2}{d+2}C_d\left(\frac{k}{|\Omega|}\right)^{2/d}}\frac{|\Omega|}{|\partial\Omega|}
$$
leads to the final estimate,
$$
  \frac{1}{k}\sum_{j=1}^{k}\lambda_j\leq \frac{d}{d+2}\,C_d\left(\frac{k}{|\Omega|}\right)^{\frac{2}{d}}+2\left(\frac{2}{d+2}\,C_d\right)^{\frac{1}{2}}\left(\frac{k}{|\Omega|}\right)^{\frac{1}{d}}\frac{|\partial \Omega|}{|\Omega|}+\frac{4|\partial \Omega|^2}{|\Omega|^2}.
$$
This proves \eqref{evsums-DirichletLaplacian15_v30-NEW}.

In order to conclude the proof, we note that the first eigenvalue $\lambda_1$ satisfies the variational inequality \eqref{up_bd_lambda1}.
Hence choosing the function $\phi=\phi_h$ in \eqref{up_bd_lambda1}, we deduce that
\begin{equation}\label{planar_convex_lambda1}
\lambda_1\leq\frac{1}{h}\frac{|\partial\Omega|}{|\Omega|-h|\partial\Omega|}.
\end{equation}
Since $r_{\Omega}\geq\frac{|\Omega|}{|\partial\Omega|}$, we can choose any $h\leq \frac{|\Omega|}{|\partial\Omega|}$ in \eqref{planar_convex_lambda1}. By taking $h=\frac{|\Omega|}{2|\partial\Omega|}$ we obtain \eqref{planar_convex_lambda12}. This concludes the proof of \eqref{planar_convex_lambda12} and of the theorem.
\end{proof}

\begin{rem}
We note that a simple adaptation of the proof of inequality \eqref{evriesz-DirichletLaplacian15_v30-NEW} shows that for all $\alpha>0$ and for all $z\geq \frac{\alpha d}{r_{\Omega}^2}$,
\begin{equation}\label{AAA}
  \sum_{j}(z-\lambda_j)_{+}\geq \frac{2}{d+2}\,(2\pi)^{-d}\omega_d |\Omega|z^{\frac{d}{2}+1}-(2\pi)^{-d}\omega_d |\partial\Omega|z^{\frac{d}{2}+\frac{1}{2}}\left(\frac{1}{\sqrt{\alpha d}}+2\sqrt{\alpha d}\right).
\end{equation}
The left side of \eqref{AAA} is zero whenever $z\leq\lambda_1$. In this regard we mention the inequality
$$
\lambda_1\geq\frac{\pi^2}{4r{\Omega}^2}
$$
which holds for any convex domain of $\mathbb R^d$ (see \cite{protter2}). In particular this implies that when $d=2$ inequality \eqref{evriesz-DirichletLaplacian15_v30-NEW} holds for all $z\geq \lambda_1$.
\end{rem}



\subsection{Bounds for various choices of \texorpdfstring{$\phi$}{phi} approximating the characteristic function of regular domains: proofs}\label{sub:2:2}

Before proving the results contained in Subsection \ref{sec:1.5.2} it is worth recalling the following definition.

\begin{defn}\label{defn}
Let $E\subset\mathbb R^d$ be a closed set. The upper and lower outer Minkowski contents $\mathcal{M}^+(E)$ and $\mathcal{M}^-(E)$ are defined respectively as
$$
\mathcal{M}^+(E)=\limsup_{h\rightarrow 0^+}\frac{|E^h\setminus E|}{h}\ \ \ \ \ {\rm and}\ \ \ \ \ \mathcal{M}^-(E)=\liminf_{h\rightarrow 0^+}\frac{|E^h\setminus E|}{h},
$$
where $E^h:=\left\{x\in\mathbb R^d:\text{dist}(x,E)\leq h\right\}$. If $\mathcal{M}^+(E)=\mathcal{M}^-(E)<\infty$, we denote by $\mathcal{M}(E)$ their common value and we say that $E$ admits outer Minkowski content $\mathcal{M}(E)$.
\end{defn}
By definition,
\begin{equation}\label{limit}
\lim_{h\rightarrow 0^+}\frac{|\omega_h|}{h}=: \mathcal{M}(\mathbb R^d\setminus\Omega).
\end{equation}
The limit \eqref{limit} is often also called the Minkowski content of $\partial\Omega$ {\it relative to} $\Omega$ (see e.g., \cite{lapidus_fractal_1,lapidus_fractal_2}).

 As mentioned earlier, if the boundary $\partial\Omega$ is sufficiently smooth, then the limit \eqref{limit} gives $|\partial\Omega|$. For example, we have the following proposition.
\begin{proposition}\label{lipschitzprop}
If $\Omega$ is a compact subset of $\mathbb R^d$ with Lipschitz boundary, then
$$
\lim_{h\rightarrow 0^+}\frac{|\omega_h|}{h}=|\partial\Omega|.
$$
\end{proposition}
We refer to \cite{colesanti_ambrosio} for the proof and for a more detailed discussion of the outer Minkowski content and for conditions on sets $E$ ensuring that $\mathcal{M}(E)=|\partial E|$.

We are now ready to prove Theorem \ref{main_thm_evsums-DirichletLaplacian}.

\begin{proof}[Proof of Theorem \ref{main_thm_evsums-DirichletLaplacian}]
As in the proof of inequality \eqref{evsums-DirichletLaplacian12-NEW},  it follows
from \eqref{evsums-DirichletLaplacian1} that
\begin{equation}\label{evsums-DirichletLaplacian12}
 \frac1{k}\sum_{j=1}^k\lambda_{j}\leq \frac{d}{d+2}\,C_d\left(\frac{k}{|\Omega|}\right)^{\frac{2}{d}}+\frac{2}{d+2}\,C_d\left(\frac{k}{|\Omega|}\right)^{\frac{2}{d}}\left(\frac{|\Omega|\cdot||\phi||_{\infty}^2}{||\phi||_2^2}-1\right)+\frac{||\nabla\phi||_2^2}{||\phi||_2^2}
\end{equation}
for all $\phi\in H^1_0(\Omega)\cap L^{\infty}(\Omega)$. In order to get an estimate we need a suitable choice for $\phi$ and upper bounds on $\displaystyle \frac{|\Omega|\cdot||\phi||_{\infty}^2}{||\phi||_2^2}$ and $\displaystyle \frac{||\nabla\phi||_2^2}{||\phi||_2^2}$. We choose
$\phi=\phi_h$ in \eqref{evsums-DirichletLaplacian12}, where $\phi_h$ is defined by \eqref{phi-h-test-fct} with $f(p)=p$. Thanks to \eqref{up_2_phi_h} and \eqref{up_2_nabla_phi_h}, inequality \eqref{evsums-DirichletLaplacian12} becomes
\begin{equation}\label{evsums-DirichletLaplacian13}
 \frac1{k}\sum_{j=1}^k\lambda_{j}\leq \frac{d}{d+2}\,C_d\left(\frac{k}{|\Omega|}\right)^{\frac{2}{d}}+\left(\frac{2}{d+2}\,C_d\left(\frac{k}{|\Omega|}\right)^{\frac{2}{d}}+\frac{1}{h^2}\right)\cdot\frac{|\omega_h|}{|\Omega|-|\omega_h|}
\end{equation}
for all $h\leq r_{\Omega}$. Formula \eqref{evsums-DirichletLaplacian13} holds in great generality under no regularity assumptions on the domain.

Now suppose that $\Omega\in\mathcal S$.  We can rewrite \eqref{evsums-DirichletLaplacian13} as follows:

\begin{equation}\label{evsums-DirichletLaplacian144}
\frac1{k}\sum_{j=1}^k\lambda_{j}\leq \frac{d}{d+2}\,C_d\left(\frac{k}{|\Omega|}\right)^{\frac{2}{d}}+\left(\frac{2}{d+2}\,C_d\left(\frac{k}{|\Omega|}\right)^{\frac{2}{d}}h+\frac{1}{h}\right)\cdot\frac{|\partial\Omega|}{|\Omega|}+R_k(h),
\end{equation}
where
\begin{equation}\label{remainderR}
R_k(h)=\left(\frac{2}{d+2}C_d\left(\frac{k}{|\Omega|}\right)^{\frac{2}{d}}+\frac{1}{h^2}\right)\cdot\frac{h|\partial\Omega||\omega_h|+|\Omega|(|\omega_h|-h|\partial\Omega|)}{|\Omega|(|\Omega|-|\omega_h|)}.
\end{equation}
We neglect for the moment the term $R_k(h)$ in \eqref{evsums-DirichletLaplacian144} and optimize the second summand with respect to $h$. The expression
$$
\left(\frac{2}{d+2}\,C_d\left(\frac{k}{|\Omega|}\right)^{\frac{2}{d}}h+\frac{1}{h}\right)\cdot\frac{|\partial\Omega|}{|\Omega|}
$$
is optimized when
\begin{equation}\label{hk}
h=h(k):=\left(\frac{2}{d+2}C_d\right)^{-\frac{1}{2}}\left(\frac{k}{|\Omega|}\right)^{-\frac{1}{d}}.
\end{equation}
By using \eqref{hk} in \eqref{evsums-DirichletLaplacian144} and the fact that $\Omega\in\mathcal S$,  \eqref{evsums-DirichletLaplacian14} follows
immediately (we set $R(k):=R_k(h(k))$). 
\end{proof}

\begin{rem}[Domains with fractal boundary]
The proof of Theorem \ref{main_thm_evsums-DirichletLaplacian} can be adapted to more general situations, in particular to the case of fractal boundaries. In this connection, we mention the famous Weyl-Berry conjecture, which states  that in the case of a bounded domain $\Omega$, if $\partial\Omega$ has Hausdorff dimension $H$, then $N(\lambda)-(2\pi)^{-d}\omega_d|\Omega|\lambda^{d/2}$ is asymptotically a constant times $\lambda^{H/2}$, where the constant is proportional to the normalized Hausdorff measure of the boundary. Here $N(\lambda)$ denotes the counting function of the Dirichlet Laplacian on $\Omega$. The conjecture in this form is false (see \cite{carmona_fractal}), and in \cite{lapidus_fractal_1} it is conjectured that if $D\in]d-1,d[$ then
$$
N(\lambda)=(2\pi)^{-d}\omega_d|\Omega|\lambda^{d/2}-c_{n,D}\mathcal M_D(\partial\Omega)\lambda^{D/2}+o(\lambda^{D/2})
$$
as $\lambda\rightarrow+\infty$,
where
$$
D:=\inf\left\{\gamma\in[d-1,d]:\lim_{h\rightarrow 0^+}\frac{|\omega_h|}{h^{d-\gamma}}<\infty\right\}
$$
is the Minkowski dimension of $\partial\Omega$ relative to $\Omega$ and
$$
\mathcal M_D(\partial\Omega):=\lim_{h\rightarrow 0^+}\frac{|\omega_h|}{h^{d-D}}
$$
is the $D$-dimensional Minkowski content of $\partial\Omega$ relative to $\Omega$.
This conjecture, however,
was likewise revealed to be false except in the case $d=1$ (see \cite{lapidus_fractal_2}). In fact it is proved in \cite{lapidus_fractal_2} that the spectrum depends not merely on $d$, $D$, $|\Omega|$, and $\mathcal M_D(\partial\Omega)$, but on additional geometry.  It is however important to remark that if $D\in]d-1,d[$ is such that $\mathcal M_D(\partial\Omega)<+\infty$, then $N(\lambda)-(2\pi)^{-d}\omega_d|\Omega|\lambda^{d/2}=O(\lambda^{D/2})$ as $\lambda\rightarrow +\infty$.  (Actually one only needs a $D\in]d-1,d[$ such that $\lim\sup_{h\rightarrow 0^+}h^{-(d-D)}|\omega_h|<+\infty$; see \cite[Theorem 2.1]{lapidus_fractal_1}). Therefore the Minkowsi dimension of $\partial\Omega$ relative to $\Omega$ determines the order of the correction in the asymptotic formula of the counting function. In particular this fact implies the following asymptotic formula for sums with sharp remainder:
$$
\frac{1}{k}\sum_{j=1}^k\lambda_j=\frac{d}{d+2}C_d\left(\frac{k}{|\Omega|}\right)^{\frac{2}{d}}+O(k^{D-d+2})
$$
as $k\rightarrow +\infty$. Assume now that $\Omega$ is a bounded domain such that the Minkowski dimension relative to $\Omega$ of $\partial\Omega$ is $D\in]d-1,d[$ and let $\mathcal M_D(\partial\Omega)$  be the Minkowski content of $\partial\Omega$ relative to $\Omega$. By following the steps of the proof of Theorem \ref{main_thm_evsums-DirichletLaplacian} one immediately obtains
\begin{multline*}
\frac{1}{k}\sum_{j=1}^{k}\lambda_j\leq \frac{d}{d+2}C_d\left(\frac{k}{|\Omega|}\right)^{\frac{2}{d}}\\+2\left(\frac{2C_d(d-D)}{(d+2)(D-d+2)}\right)^{\frac{D-d+2}{2}}\frac{\mathcal M_D(\Omega)}{|\Omega|}\left(\frac{k}{|\Omega|}\right)^{\frac{D-d+2}{d}}+R(k),
\end{multline*}
where $R(k)=o(k^{\frac{D-d+2}{d}})$ as $k\rightarrow +\infty$,  Hence the second term of the upper bound for the sum depends only on $k$, $d$, $D$, $|\Omega|$ and $\mathcal M_D(\Omega)$.
\end{rem}

We conclude this subsection with the proof of Corollary \ref{thm_bounds_part_func_regular}.

\begin{proof}[Proof of Corollary \ref{thm_bounds_part_func_regular}]
We use the test function $\phi=\phi_h$ as in \eqref{phi-h-test-fct} with $f(p)=p$ in the lower bound \eqref{part-fct-estimate-small-times} for the partition function to get, for all $h\leq r_{\Omega}$,
\begin{multline*}
  \sum_{j=1}^{\infty}e^{-\lambda_{j}t}\geq \frac{|\Omega|}{(4\pi t)^{\frac{d}{2}}}-\frac{1}{(4\pi t)^{\frac{d}{2}}}\left(\frac{t|\omega_h|}{h^2}+\int_{\omega_h}\left(1-\frac{\delta^2(x)}{h}\right)\,dx\right)\\
	\geq\frac{|\Omega|}{(4\pi t)^{\frac{d}{2}}}-\frac{|\omega_h|}{(4\pi t)^{\frac{d}{2}}}\left(\frac{t}{h^2}+1\right)\\
	=\frac{|\Omega|}{(4\pi t)^{\frac{d}{2}}}-\frac{|\partial\Omega|}{(4\pi t)^{\frac{d}{2}}}\left(\frac{t}{h}+h\right)+\frac{\left(h|\partial\Omega|-|\omega_h|\right)}{(4\pi t)^{d/2}}\left(\frac{t}{h^2}+1\right).
\end{multline*}
Choosing $\displaystyle h=\sqrt{t}$ and recalling that $\Omega\in\mathcal S$, the result immediately follows. 
\end{proof}




\subsection{Estimates for domains of class \texorpdfstring{$C^2$}{C2}: proofs}\label{sub:2:3}

As we shall see throughout this subsection, computations become more explicit if $\Omega$ is of class $C^2$. Before proving the main results of Subsection \ref{sec:1.5.3} (namely, Theorem \ref{smooth_1} and Corollary \ref{smooth_convex}), we need to recall some useful results on the tubular neighbourhood of the boundary of a $C^2$ domain. The $h$-tubular neighbourhood $\omega_h$ of $\partial\Omega$ was defined in \eqref{tubular_n}.

\begin{theorem}\label{tubular0}
Let $\Omega$ be a bounded domain in $\mathbb R^{d}$ of class $C^2$. Then there exists $h>0$ such that every point in $\omega_h$ has a unique nearest point on $\partial\Omega$. Moreover, the function $\delta$ is of class $C^2$ in $\omega_h$.
\end{theorem}
We refer to \cite{krantz} for the proof of Theorem \ref{tubular0}. (See also \cite[Ch.6, Theorem 6.3]{shapes} and \cite[Lemma 14.16]{gitr}.) Throughout the rest of the paper we shall denote by $\bar h$ the maximal possible tubular radius of $\Omega$, which has been defined in \eqref{barh}. From Theorem \ref{tubular0} it follows that if $\Omega$ is of class $C^2$ then such an $\bar h$ exists and is positive.

Throughout the rest of this section, we denote by $h$ a positive number such that $0<h<\bar h$.

Let $x\in\partial\Omega$ and let $\kappa_1(x),...,\kappa_{d-1}(x)$ denote the principal curvatures of $\partial\Omega$ at $x$ with respect to the outward unit normal. We refer e.g., to \cite[Sec. 14.6]{gitr} for the definition and basic properties of the principal curvatures of $\partial\Omega$. In particular if $y\in\omega_{{h}}$ and $x\in\partial\Omega$ is the nearest point to $x$ on $\partial\Omega$, then
\begin{equation}\label{bound_curv}
1-\delta(x)\kappa_i(y)>0
\end{equation}
for all $i=1,...,N$ (see e.g., \cite[Lemma 2.2]{lewis}). 

The mean curvature $\mathcal H(x)$ of $\partial\Omega$ at a point $x$ is defined as
\begin{equation*}
\mathcal H(x):=\frac{1}{d-1}\sum_{i=1}^{d-1}\kappa_i(x).
\end{equation*}

We are now ready to prove Theorem \ref{smooth_1} and Corollary \ref{smooth_convex}.

\begin{proof}[Proof of Theorem \ref{smooth_1}]
Since $\Omega$ is of class $C^2$, in particular $\Omega$ belongs to the class $\mathcal S$ (see Proposition \ref{lipschitzprop}), hence Theorem \ref{main_thm_evsums-DirichletLaplacian} holds. 
Now we estimate $|\omega_h|$ for small $h$. Let $0<h<\bar h$. We shall denote by $s$ an element of $\partial\Omega$ viewed as an embedded $d-1$-dimensional manifold. We denote the induced metric on $\partial\Omega$ by $g_s$ and the induced $d-1$-dimensional volume form by $d\sigma(s)=|\det g_s|^{1/2}$. For $s\in \partial\Omega$, we shall denote by $\nu(s)$ the outward unit normal to $\partial\Omega$. It is well-known that the map $\Phi$ defined by
$$
\Sigma:=\partial\Omega\times (0,h)\ni(s,t)\mapsto\Phi(s,t)=s-t\nu(s)\in\omega_h
$$
is a diffeomorphism.  (See e.g., \cite[Sec. 2.4]{balinskybook}, and see also Theorem \ref{tubular0}.)
In particular, $t=\delta(s-t\nu(s))$. The metric induced by $\Phi$ on $\Sigma$ is then given by $G=g\circ (Id_{\partial\Omega}-t D\nu(s))^2+dt^2$, where $Id_{\partial\Omega}$ is the identity on the tangent space. The volume form $d\Sigma$ is given by
\begin{equation*}
d\Sigma=|\det G|^{1/2}=\prod_{i=1}^{d-1}(1-t\kappa_i(s))dtd\sigma(s).
\end{equation*}
Hence we can write
\begin{equation}\label{wh_coordinates}
|\omega_h|=\int_{\partial\Omega}\int_0^h\prod_{i=1}^{d-1}(1-t\kappa_i(s))dtd\sigma(s).
\end{equation}
Thanks to \eqref{bound_curv},
$$
\prod_{i=1}^{d-1}(1-t\kappa_i(s))\leq\left(\frac{1}{d-1}\sum_{i=1}^{d-1}(1-t\kappa_i(s))\right)^{d-1}=\left(1-t\mathcal H(s)\right)^{d-1}.
$$
Hence
\begin{equation}\label{est0}
|\omega_h|\leq\int_{\partial\Omega}\int_0^h\left(1-t\mathcal H(s)\right)^{d-1}dtd\sigma(s).
\end{equation}
Integrating by parts twice we see that
\begin{multline}\label{est00}
\int_0^h(1-t\mathcal H(s))^{d-1}dt\\
=h-\frac{d-1}{2}\mathcal H(s)h^2+\frac{(d-1)(d-2)}{2}\mathcal H^2(s)\int_0^h(t-h)^2(1-t\mathcal H(s))^{d-3}dt.
\end{multline}
Using $(t-h)^2\leq h^2$ in the integral on the right side of \eqref{est00} and integrating with respect to $t$ we get the inequality
\begin{equation*}
\int_0^h(1-t\mathcal H(s))^{d-1}dt\leq h-\frac{d-1}{2}\mathcal H(s)h^2(1-h\mathcal H(s))^{d-2}
\end{equation*}
and consequently
\begin{equation}\label{est000}
|\omega_h|\leq h|\partial\Omega|-\frac{d-1}{2}h^2\int_{\partial\Omega}\mathcal H(s)(1-h\mathcal H(s))^{d-2}d\sigma(s),
\end{equation}
with equality if $d=2$.

From \eqref{est000} we deduce that there exists a constant $C'$ which depends only on $d$, $|\partial\Omega|$, $\bar h$ and $\mathcal H$ such that
\begin{equation*}
|\omega_h|-h|\partial\Omega|\leq C'h^2
\end{equation*}
for all $h\leq\bar h$. Moreover for all $\bar h\leq h\leq r_{\Omega}$, $|\omega_h|-h|\partial\Omega|\leq\frac{|\omega_{r_{\Omega}}|}{\bar h^2}h^2$. It follows immediately from the definition \eqref{remainderR} of $R_k(h)$ and \eqref{hk} that
$$
R_k(h(k))\leq C''\left(\frac{2}{d+2}C_d\left(\frac{k}{|\Omega|}\right)^{2/d}h(k)^2+1\right)=2C''
$$
for all $k\geq |\Omega|r_{\Omega}^{-d}\left(\frac{d+2}{2C_d}\right)^{d/2}$, and hence (i) 
is valid with $C:=2C''$.

We next prove (ii). From \eqref{est0} it immediately follows that
\begin{equation}\label{develop}
|\omega_h|\leq h|\partial\Omega|+\frac{h^2}{d}\sum_{j=2}^d \binom{d}{j}(-1)^{j-1}h^{j-2}\int_{\partial\Omega}\mathcal H(s)^{j-1}d\sigma(s).
\end{equation}
By plugging \eqref{develop} into \eqref{remainderR} and setting $h=h(k)$ (see \eqref{hk} for the definition of $h(k)$), \eqref{remainder-f1} immediately follows. Formula \eqref{limit_r1} follows from a standard computation. This concludes the proof of (ii) and of the theorem.
\end{proof}

\begin{proof}[Proof of Corollary \ref{smooth_convex}]
The proof of point (i) is identical to that of point (i) of Theorem \ref{smooth_1} and is accordingly omitted. The proof of point (ii) is similar to that of \eqref{evsums-DirichletLaplacian15_v30-NEW} in Theorem \ref{thm_convex_first_eigenvalues}. First we note that from \eqref{est000} it follows that for all $h<\bar h$, $|\omega_h|\leq h|\partial\Omega|$. Hence inequality \eqref{evsums-DirichletLaplacian13-NEW-2} holds with $X=\frac{h|\partial\Omega|}{|\Omega|}$, for all $h<\bar h$. Then in \eqref{evsums-DirichletLaplacian13-NEW-2} we insert the optimal $X_0$ given by \eqref{optX0}, which is admissible only if $X_0\leq\frac{\bar h|\partial\Omega|}{|\Omega|}$, that is, if $k\geq|\Omega|\left(\frac{d+2}{2C_d}\frac{|\Omega|-2\bar h|\partial\Omega|}{\bar h^2|\Omega|}\right)^{\frac{d}{2}}$. This proves point (ii). As for point (iii), we note that if $\bar h\geq\frac{|\Omega|}{2|\partial\Omega|}$, then $X_0$ defined by \eqref{optX0} is admissible independently of $k$, that is, $X_0\leq\frac{\bar h|\partial\Omega|}{|\Omega|}$ for all positive integer $k$. This concludes the proof of point (iii) and of the corollary.

\end{proof}




\subsection{Estimates for planar sets: proofs}\label{sub:2:4}

In this subsection we present the proofs of the bounds contained in Theorem \ref{planar_thm}.

\begin{proof}[Proof of Theorem \ref{planar_thm}]
We start by proving (i). For a planar $C^2$ domain and $0<h<\bar h$, \eqref{wh_coordinates} immediately implies that
\begin{equation*}
|\omega_h|=h|\partial\Omega|-\frac{h^2}{2}\int_{\partial\Omega}\kappa(s)ds,
\end{equation*}
where $ds$ is the arc-length element and $\kappa(s)$ denotes the curvature of $\partial\Omega$ (the orientation is chosen according to the outer unit normal to $\partial\Omega$). Moreover, for a closed curve $\gamma$, the quantity $\int_{\gamma}\kappa(s)ds$ is the total curvature of the curve.
This quantity is in particular an integer multiple of $2\pi$, namely the winding number of the unit tangent vector about the origin. It is straightforward to see that
\begin{equation*}
\int_{\partial\Omega}\kappa(s)ds=2\pi(2-b),
\end{equation*}
where $b$ denotes the number of connected components of $\partial\Omega$. Hence for all $0<h<\bar h$,
\begin{equation}\label{wh_planar_C2}
|\omega_h|=h|\partial\Omega|-\pi(2-b)h^2.
\end{equation}
Formula \eqref{remainder_planarC2} follows by plugging \eqref{wh_planar_C2} into \eqref{remainderR} and by taking $h=h(k)$ (see \eqref{hk}). This concludes the proof of (i).

The proof of (ii) is identical to that of point (ii) of Corollary \ref{smooth_convex}; in fact in the case that $\partial\Omega$ has one or two connected components, then $|\omega_h|\leq h|\partial\Omega|$.

Point (iii) is a straightforward application of formula \eqref{evsums-DirichletLaplacian15_v30-NEW}.

Consider now (iv). We can assume without loss of generality that $\partial\Omega$ is connected. Let $n:=n_a+n_b$ be the number of angles of $\Omega$, and let $\left\{l_1,...,l_n\right\}$ denote the angle bisectors of consecutive angles. Let $o_i:=l_i\cap l_{i+1}$ ($o_n=l_n\cap l_1$) if such intersection is non-empty and let $O:=\left\{o_i\right\}_{i=1}^n$ (where we agree to delete from the set the element $o_i$ if $l_i\cap l_{i+1}$ is empty).  Then let
\begin{equation*}
\tilde h:=\min_{o_i\in O\cap\Omega}\delta(o_i).
\end{equation*}
Assume now that $0<h<\tilde h$. We want to compute $|\omega_h|$. A first approximation is clearly $|\omega_h|=h|\partial\Omega|+o(h)$. Let us denote by $s_1,...,s_n$ the length of the sides of the polygon. Then $h|\partial\Omega|=h (s_1+...+s_n)$ is the sum of the areas of $n$ rectangles with side lengths $h$ and $s_i$, $i=1,...,n$.  Under the assumptions that $0<h<\tilde h$, we note that in considering an angle $\alpha$ with $0<\alpha<\pi$ we have taken into account an additional portion of area which measures $h^2\cot(\alpha/2)$.  Conversely, in considering an angle $\beta$ with $\pi<\beta<2\pi$ we did not take into account a portion of the tubular neighbourhood, which is a circular sector of radius $h$ and width $\beta-\pi$.  Hence we have to add the corresponding area, namely $h^2(\beta-\pi)/2$. If then $\left\{\alpha_i\right\}_{i=1}^{n_a}$ denotes the set of the angles between $0$ and $\pi$ and $\left\{\beta_i\right\}_{i=1}^{n_b}$ denotes the set of the angles between $\pi$ and $2\pi$, we have
\begin{equation}\label{wh_polygonal}
|\omega_h|=h|\partial\Omega|-h^2\sum_{i=1}^{n_a}\cot(\alpha_i/2)+h^2\sum_{i=1}^{n_b}\frac{\beta_i-\pi}{2}.
\end{equation}
Plugging \eqref{wh_polygonal} into \eqref{remainderR} with $h=h(k)$
immediately yields \eqref{remainder_polygon}. This concludes the proof of (iv) and of the theorem.
\end{proof}




\section{Application of the AVP to the Neumann\texorpdfstring{\\}{ }Laplacian: proofs of the main results}\label{sec:3}

In this section we present the proofs of Theorems \ref{refinement_kroger_thm} and \ref{thm_bounds_neumann_eigenvalues}.





\begin{proof}[Proof of Theorem \ref{refinement_kroger_thm}]
We note that by applying the averaged variational principle \eqref{averaged-var-principle-Rieszmean2_0} to the Neumann problem, we obtain the analogue of Theorem \ref{dirichlet_laplacian_thm_general} where $\lambda_{k+1},\lambda_j$ are replaced by $\mu_{k+1},\mu_j$. Moreover the test function $\phi$ in Theorem \ref{dirichlet_laplacian_thm_general} can be chosen in $H^1(\Omega)\cap L^{\infty}(\Omega)$. The most natural choice is $\phi\equiv 1$. We deduce then that the averaged variational principle applied to the Neumann Laplacian provides an efficient way to recover Kr\"oger's result, as noted in \cite{EHIS16}. In fact \eqref{kroger_0} follows from \eqref{evsums-DirichletLaplacian1} if we replace $\lambda_j$ by $\mu_j$ and take $\phi\equiv 1$. Then, roughly speaking, Theorem \ref{refinement_kroger_thm} is a corollary of Theorem \ref{dirichlet_laplacian_thm_general} for Neumann eigenvalues.

The proof of \eqref{Neumann-Laplacian-ev-bound-normalized2} is actually contained in \cite[Theorem 1.1]{harrell_stubbe_18}. We note that \eqref{Neumann-Laplacian-ev-bound-normalized2} can be also obtained as a consequence of \eqref{general_AVP-DirichletLaplacian} in the proof of Theorem \ref{dirichlet_laplacian_thm_general} if we replace $\lambda_{k+1},\lambda_j$ with $\mu_{k+1},\mu_j$ and take $\phi\equiv 1$, and by applying a refinement
of Young's inequality (see \cite[Appendix A]{harrell_stubbe_18}). Inequalities \eqref{Neumann-Laplacian-ev-bound-normalized3} follow  immediately from Corollary \ref{quadratic-eigenvalue-inequality-general} if we replace $\lambda_k$ by $\mu_k$ and take $\phi\equiv 1$ (hence $\rho(\phi)\equiv 1$). Inequality \eqref{riesz_means_neumann} follows from \eqref{Riesz-mean-ineq-DirichletLaplacian} if we replace $\lambda_j$ by $\mu_j$ and take $\phi\equiv 1$. For the partition function, we obtain \eqref{part_func_neumann} by replacing $\lambda_j$ by $\mu_j$ and using  $\phi\equiv 1$ in \eqref{part-fct-estimate-small-times}. In particular we observe that a lower bound for the trace of the Neumann heat kernel is given by the partition function of the free particle.
\end{proof}





We now turn our attention to the proof of Theorem \ref{thm_bounds_neumann_eigenvalues}. In order to obtain asymptotically sharp lower bounds for means of Neumann eigenvalues, we shall exploit the averaged variational principle applied to the Dirichlet Laplacian as in Section \ref{sec:2} with test functions given in terms of Neumann eigenfunctions. Doing so, we obtain lower bounds for the Riesz mean $\sum_{j}(z-\lambda_j)_+$ in terms of the Riesz mean $\sum_{j}(z-\mu_j)_+$ with a remainder of the correct order. Then we use the semiclassically sharp upper bounds for Riesz means of the Dirichlet Laplacian to obtain upper bounds for Riesz means of Neumann Laplacian, which turn out to be equivalent to lower bounds for averages.  In the remainder of this section $\Omega$ is assumed to be a bounded domain in $\mathbb R^d$ of class $C^2$, $\bar h$ denotes the maximal possible size of a tubular neighbourhood about $\partial\Omega$ (see \eqref{barh}), and $\kappa_i(x)$, $i=1,...,d-1$, denotes the principal curvatures at $x\in\partial\Omega$.

As already mentioned, in order to prove Theorem \ref{thm_bounds_neumann_eigenvalues} we will apply the averaged variational principle with trial functions of the form $f_j(x)=\phi_h(x)v_j(x)$ with $\phi_h$ as in \eqref{phi-h-test-fct} and $v_j$ the Neumann eigenfunctions. This indicates that we deal with the integral of the spectral function $\sum_{j=1}^k v_j^2$ in $\omega_h$. As we shall see, we need uniform control on the spectral function near the boundary. Before proving Theorem \ref{thm_bounds_neumann_eigenvalues} we recall some preliminary results.

The first result provides an estimate of the $L^{\infty}$ norm of the spectral function far from the boundary.
\begin{theorem}\label{thm_safarov_spectral_function}
Let $\Omega$ be a bounded domain in $\mathbb R^d$ such that $H^1(\Omega)\subset L^2(\Omega)$ is compact. Then for all $\mu>0$,
\begin{multline}\label{ineq_spec_func}
\sum_{\mu_j\leq\mu} v_j^2(x)\leq (2\pi)^{-d}\omega_d\mu^{d/2}\\
+\frac{d(d+2)(2\pi)^{-d}\sqrt[d+1]{3}\omega_d(2\pi^{-1}(d+2)\sqrt[d+1]{3}+1)}{\delta(x)}\left(\mu^{1/2}+\frac{(d+2)\sqrt[d+1]{3}}{\delta(x)}\right)^{d-1}.
\end{multline}
\end{theorem}
Theorem \ref{thm_safarov_spectral_function} follows from \cite[Corollary 3.1]{safarov_tauberian}. In order to prove Theorem \ref{thm_bounds_neumann_eigenvalues} we also need to control the $L^{\infty}$ norm of the spectral function in a tubular neighbourhood of the boundary of the size $O(k^{-1/d})$. We recall the following theorem.
\begin{theorem}
Let $\Omega$ be a bounded domain in $\mathbb R^d$ of class $C^2$. Let $\alpha>\frac{1}{2}$ and $0<\beta<\frac{1}{\sqrt{\alpha}}$ be fixed real numbers. Then for all $\mu>0$ such that
\begin{equation}\label{lower_mu}
\mu\geq\max\left\{ 4\beta^2 \bar h^{-2},\frac{4\alpha^2\beta^2}{(2\alpha-1)^2}\max_{\substack{x\in\partial\Omega\\0\leq h\leq\bar h/2}}\left|\sum_{i=1}^{d-1}\frac{h\kappa_i(x)}{1-h\kappa_i(x)}\right|^2\right\},
\end{equation}
we have
\begin{equation}\label{xu_sogge}
\max_{\left\{x\in\Omega:\delta(x)\leq \frac{\beta}{\sqrt{\mu}}\right\}}\sum_{\mu_j\leq\mu}v_j(x)^2\leq\frac{1}{(1-\alpha\beta^2)^2}\max_{\left\{x\in\Omega:\delta(x)=\frac{\beta}{\sqrt{\mu}}\right\}}\sum_{\mu_j\leq\mu}v_j(x)^2.
\end{equation}
\end{theorem}
The proof can be carried out with very few modifications in the same way as that of Proposition 2.2 of \cite{xu_bdry} (see also \cite{sogge_bdry}). We remark that lower bounds for values of $\mu$ for which \eqref{xu_sogge} holds depend on $\bar h$ and on $L^{\infty}$ estimates on the coefficients of lower order terms in the expression of the Laplace operator in local coordinates inside a tubular neighbourhood. In particular, for a domain of class $C^2$, if $\partial\Omega_h=\left\{x\in\Omega:\delta(x)=h\right\}$ denotes the inner $h$-parallel set of $\partial\Omega$ , then for all $0<h<\bar h$ and all $x\in\omega_{\bar h}$ with $\delta(x)=h$ we have
\begin{equation}\label{local}
\Delta u(x)=\Delta_{\partial\Omega_h}u(x)+\sum_{i=1}^{d-1}\frac{\kappa_i(y)\delta(x)}{1-\kappa_i(y)\delta(x)}\frac{\partial u}{\partial\nu}(x)+\frac{\partial^2u}{\partial\nu^2}(x),
\end{equation}
where
$y$ denotes the nearest point on $\partial\Omega$ to $x$, $\nu$ is the outer unit normal to $y$, and
$\Delta_{\partial\Omega_h}$ denotes the Laplace-Beltrami operator on $\partial\Omega_h$. In particular, exploiting formula \eqref{local} in the proof of Proposition 2.2 of \cite{xu_bdry} (and hence in the proof of \eqref{xu_sogge}) yields the explicit bound \eqref{lower_mu}.

It is also useful to state the following lemma.
\begin{lemma} Let $u,\phi:\Omega\rightarrow \mathbb{C}$ such that $-\Delta u=\lambda u$ in $L^2(\Omega)$ for some $\lambda\in\mathbb{R}$, $\phi u\in H^{1}_{0}(\Omega)$ and $\displaystyle \phi^2 u\frac{\partial\,u}{\partial\,n}$ vanishes on the boundary $\partial \Omega$. Then
\begin{multline}\label{grad-phi-u-identity}
  \int_{\Omega}|\nabla \phi u|^2\,dx\\
	=\int_{\Omega}|\nabla\phi|^2|u|^2\,dx+\lambda \int_{\Omega}|\phi u|^2\,dx+\frac{1}{2}\int_{\Omega}(\bar{\phi} \nabla \phi-\phi \nabla \bar{\phi})(u \nabla \bar{u}-\bar{u} \nabla u)\,dx.
\end{multline}
In particular, if one of the functions $u,\phi$ is real-valued, then the last term vanishes.
\begin{proof}
The proof follows by standard integration by parts and is therefore omitted.
\end{proof}
\end{lemma}

We are now ready to prove Theorem \ref{thm_bounds_neumann_eigenvalues}.

\begin{proof}[Proof of Theorem \ref{thm_bounds_neumann_eigenvalues}]

Let $\left\{\lambda_j\right\}_{j=1}^{\infty}$ be the set of Dirichlet eigenvalues on $\Omega$, repeated according their multiplicity, and $\left\{u_j\right\}_{j=1}^{\infty}$ be the corresponding orthonormal set of eigenfunctions in $L^2(\Omega)$. Let $\left\{v_j\right\}_{j=1}^{\infty}$ be the orthonormal set of Neumann eigenfunctions in $L^2(\Omega)$ associated with $\left\{\mu_j\right\}_{j=1}^{\infty}$.

We apply the averaged variational principle \eqref{averaged-var-principle-Rieszmean2_0} for the Dirichlet Laplacian with trial functions of the form $f_j(x)=\phi(x)v_j(x)$ for $j\in\mathbb N$ and $\phi\in H^1_0(\Omega)\cap L^{\infty}(\Omega)$ with $\|\phi\|_{\infty}=1$. The measure space for the averaging will be $\mathbb N$ and $d\mu$ in \eqref{averaged-var-principle-Rieszmean2_0} will be the standard counting measure on $\mathbb N$, while $d\mu_0=1_J d\mu $ where $J\subseteq\mathbb N$ and $1_J$ is the characteristic function of $J$. We obtain
\begin{equation}\label{AVP-dir-neu-1}
\sum_{j=1}^k(\lambda_{k+1}-\lambda_j)\int_{\Omega}\phi^2(x)u_j(x)^2dx
\geq\sum_{j\in J}\left(\lambda_{k+1}\int_{\Omega}\phi^2v_j^2dx-\int_{\Omega}|\nabla(\phi v_j)|^2dx\right).
\end{equation}
For the right side of \eqref{AVP-dir-neu-1} we deduce from \eqref{grad-phi-u-identity} that
\begin{multline}\label{AVP-dir-neu-2}
\sum_{j\in J}\left(\lambda_{k+1}\int_{\Omega}\phi^2v_j^2dx-\int_{\Omega}|\nabla(\phi v_j)|^2dx\right)\\
=\sum_{j\in J}\left(\lambda_{k+1}\int_{\Omega}\phi^2v_j^2dx-\int_{\Omega}|\nabla\phi |^2v_j^2dx-\mu_j\int_{\Omega}\phi^2v_j^2dx	\right)\\
\geq \sum_{j\in J}\left(\lambda_{k+1}\int_{\Omega}\phi^2v_j^2dx-\mu_j-\int_{\Omega}|\nabla\phi |^2v_j^2dx\right),
\end{multline}
where we have used the fact that since $\Omega$ is of class $C^2$, $v_j\in H^2(\Omega)$, so $-\Delta v_j=\mu_j v_j$ in $L^2(\Omega)$. For the right side of \eqref{AVP-dir-neu-1} we have, since $\|v_j\|_2^2=1$ and $\|\phi\|_{\infty}=1$,
\begin{equation}\label{AVP-dir-neu-3}
\sum_{j=1}^k(\lambda_{k+1}-\lambda_j)\int_{\Omega}\phi(x)^2u_j(x)^2dx \leq\sum_{j=1}^k(\lambda_{k+1}-\lambda_j).
\end{equation}
From \eqref{AVP-dir-neu-1}, \eqref{AVP-dir-neu-2} and \eqref{AVP-dir-neu-3} we deduce that
\begin{equation}\label{AVP-dir-neu-4}
\sum_{j=1}^k(\lambda_{k+1}-\lambda_j)\geq \sum_{j}\left(\lambda_{k+1}-\mu_j-\lambda_{k+1}\int_{\Omega}(1-\phi^2)v_j^2dx-\int_{\Omega}|\nabla\phi|^2v_j^2dx\right)_+.
\end{equation}
Since we may replace $\lambda_{k+1}$ by any $z\in [\lambda_k,\lambda_{k+1}]$, the bound \eqref{AVP-dir-neu-4} is equivalent to the following upper bound on the Neumann Riesz mean.
\begin{equation*}
\sum_{j}(z-\mu_j)_+\leq \sum_{j}(z-\lambda_j)_+ +z\sum_{\mu_j\leq z}\int_{\Omega}(1-\phi^2)v_j^2dx+\sum_{\mu_j\leq z}\int_{\Omega}|\nabla\phi|^2v_j^2dx.
\end{equation*}
Let $\phi=\phi_h$ where $\phi_h$ is defined by \eqref{phi-h-test-fct} with $f(p)=\sin(2p/\pi)$. A straightforward computation shows that
\begin{multline}\label{riesz-neu-2}
\sum_{j}(z-\mu_j)_+\leq \sum_{j}(z-\lambda_j)_+ +\left(z+\frac{\pi^2}{4h^2}\right)|\omega_h|\left(\max_{x\in\bar\omega_h}\sum_{\mu_j\leq z} v_j^2(x)\right)\\
= \sum_{j}(z-\lambda_j)_+ +\left(zh+\frac{\pi^2}{4h}\right)|\partial\Omega|\left(\max_{x\in\bar\omega_h}\sum_{\mu_j\leq z} v_j^2(x)\right) \\+\left(z+\frac{\pi^2}{4h^2}\right)(|\omega_h|-h|\partial\Omega|)\left(\max_{x\in\bar\omega_h}\sum_{\mu_j\leq z} v_j^2(x)\right).
\end{multline}
To optimize $\left(zh+\pi^2/4h\right)$ with respect to $h$, we take $h=h(z)=\frac{\pi}{2\sqrt{z}}$ in \eqref{riesz-neu-2} and obtain
\begin{equation}\label{riesz-neu-3}
\sum_{j}(z-\mu_j)_+\leq \sum_{j}(z-\lambda_j)_+ +\pi\sqrt{z}|\partial\Omega|\left(\max_{x\in\bar\omega_h}\sum_{\mu_j\leq z} v_j^2(x)\right)+R'(z),
\end{equation}
where
\begin{equation*}
R'(z)=2z(|\omega_{h(z)}|-h(z)|\partial\Omega|)\left(\max_{x\in\bar\omega_h}\sum_{\mu_j\leq z} v_j^2(x)\right).
\end{equation*}
We have $h=h(z)=\frac{\pi}{2\sqrt{z}}>\frac{1}{2\sqrt{z}}$. Hence

\begin{multline}\label{split_z}
\max_{x\in\bar\omega_{h(z)}}\sum_{\mu_j\leq z} v_j^2(x)\\
=\max\left\{\max_{\left\{x\in\Omega:\delta(x)\leq\frac{1}{2\sqrt{z}}\right\}}\sum_{\mu_j\leq z} v_j^2(x),\max_{\left\{x\in\Omega:\frac{1}{2\sqrt{z}}\leq\delta(x)\leq h(z)\right\}}\sum_{\mu_j\leq z} v_j^2(x)\right\}.
\end{multline}
By hypothesis,
\begin{equation}\label{bd_z}
z\geq \max\left\{\bar h^{-2},\frac{4}{9}\max_{\substack{x\in\partial\Omega\\0\leq h\leq\bar h/2}}\left|\sum_{i=1}^{d-1}\frac{h\kappa_i(x)}{1-h\kappa_i(x)}\right|^{2}\right\},
\end{equation}
so \eqref{lower_mu} holds with $\mu=z$, $\alpha=2$, $\beta=1/2$ . Therefore, from \eqref{xu_sogge} and \eqref{ineq_spec_func} we obtain

\begin{equation*}
\max_{\left\{x\in\Omega:\delta(x)\leq \frac{1}{2\sqrt{z}}\right\}}\sum_{\mu_j\leq z}v_j(x)^2\leq 4 \max_{\left\{x\in\Omega:\delta(x)= \frac{1}{2\sqrt{z}}\right\}}\sum_{\mu_j\leq z}v_j(x)^2
\leq c_{d}z^{d/2},
\end{equation*}
where
\small
\begin{equation*}
c_{d}\\
=4\omega_d\left((2\pi)^{-d}+d(d+2)\sqrt[d+1]{3}\pi^{-1-d}\left(\frac{1}{2}+\sqrt[d+1]{3}(d+2)\right)^{d-1}\left(2(d+2)\sqrt[d+1]{3}+\pi\right)\right)
\end{equation*}
\normalsize
is a constant depending only on the dimension. Again from \eqref{ineq_spec_func} we deduce that
\begin{equation*}
\max_{\left\{x\in\Omega: \frac{1}{2\sqrt{z}}\leq\delta(x)\leq h(z)\right\}}\sum_{\mu_j\leq z}v_j(x)^2\leq\frac{c_{d}}{4}z^{d/2},
\end{equation*}
and then from \eqref{split_z} we conclude that
\begin{equation}\label{est_sp_fcn_z}
\max_{x\in\omega_{h(z)}}\sum_{\mu_j\leq z} v_j^2(x)\leq c_dz^{d/2}.
\end{equation}
Plugging \eqref{est_sp_fcn_z} into \eqref{riesz-neu-3}, we obtain
\begin{equation}\label{riesz-neu-4}
\sum_{j}(z-\mu_j)_+\leq \sum_{j}(z-\lambda_j)_+ +\pi|\partial\Omega|c_dz^{\frac{d}{2}+\frac{1}{2}}+R'(z)
\end{equation}
with
$$
|R'(z)|\leq 2z^{1+\frac{d}{2}}c_d||\omega_{h(z)}|-h(z)|\partial\Omega||.
$$
and
$$
\lim_{z\rightarrow+\infty}\frac{|R'(z)|}{z^{d/2}}\leq\frac{\pi^2(d-1)c_d}{4}\left|\int_{\partial\Omega}\mathcal H(x)d\sigma(x)\right|.
$$
Using the sharp semiclassical estimate for Dirichlet Riesz means
\begin{equation*}
\sum_{j}(z-\lambda_j)_+\leq\frac{2}{d+2}C_d^{-\frac{d}{2}}|\Omega|z^{1+\frac{d}{2}}
\end{equation*}
(see \cite{Berezin}), inequality \eqref{riesz-neu-4} assumes the more explicit form \eqref{riesz-neu-5} for all $z$ satisfying \eqref{bd_z}. This concludes the proof of the bound on Riesz means.

A lower bound for sums is obtained by Legendre transforming \eqref{riesz-neu-5}. In fact, for each $k\in\mathbb N$, $k\geq 1$,
\begin{multline}\label{leg}
\sup_{z\geq z_0}\left(kz-\sum_j(z-\mu_j)_+\right)\\
\geq\sup_{z\geq z_0}\left(kz-\frac{d}{d+2}C_d^{-\frac{d}{2}}|\Omega|z^{1+\frac{d}{2}}-\pi|\partial\Omega|c_dz^{\frac{d}{2}+\frac{1}{2}}-R'(z)\right)
\end{multline}
where $z_0$ is given by the right side of \eqref{bd_z}. We note then that from \eqref{thm_safarov_spectral_function} and \eqref{est_sp_fcn_z} it follows that
$$
\sum_{\mu_j\leq z_0}\int_{\Omega}v_j(x)^2dx\leq |\Omega|c_dz_0^{d/2},
$$
and hence the number of eigenvalues smaller than $z_0$ is bounded above by\\ $|\Omega|c_dz_0^{d/2}$. This implies that if $k\geq |\Omega|c_dz_0^{d/2}$ (which is the hypothesis for \eqref{bounds_neumann_eigenvalues}), then $\mu_k>z_0$, and therefore the left side of \eqref{leg} is nothing but $\sum_{j=1}^k\mu_j$.

Now the right side of \eqref{leg} can be bounded from below by choosing an admissible $z\geq z_0$. We note that $c_d>C_d^{-d/2}$ hence any $k\geq c_d|\Omega|z_0^{d/2}$ satisfies $k\geq C_d^{-d/2}|\Omega|z_0^{d/2}$. This implies that $z=z(k)=C_d\left(\frac{k}{|\Omega|}\right)^{2/d}$ satisfies \eqref{bd_z}, that is, $z\geq z_0$. By setting $z=z(k)=C_d\left(\frac{k}{|\Omega|}\right)^{2/d}$ in the right side of \eqref{leg} we obtain
\begin{equation*}
\frac{1}{k}\sum_{j=1}^k\mu_j\geq\frac{d}{d+2}C_d\left(\frac{k}{|\Omega|}\right)^{\frac{2}{d}}-\pi c_dC_d^{\frac{d+1}{2}}\frac{|\partial\Omega|}{|\Omega|}\left(\frac{k}{|\Omega|}\right)^{\frac{1}{d}}-R(k),
\end{equation*}
where $R(k)=R'(z(k))/k$. We note that $|R'(k)|\leq C$ for some positive constant depending only on $\Omega$, and moreover
$$
\lim_{k\rightarrow+\infty}|R(k)|\leq\frac{\pi^2(d-1)c_dC_d^{\frac{d}{2}}}{4|\Omega|}\left|\int_{\partial\Omega}\mathcal H(x)d\sigma(x)\right|.
$$
This concludes the proof.
\end{proof}




\appendix
\section{Final remarks}\label{sec:4}

In this appendix we provide two final remarks.

In Appendix \ref{sub:4:1} we show that inequality \eqref{bounds_neumann_eigenvalues} can be also proved through a generalization of the Berezin-Li-Yau method \cite{Berezin,LiYau} to the Neumann problem (and more
generally to any Laplace eigenvalues).
In Appendix \ref{sub:4:2} we show how to obtain asymptotically Weyl-sharp upper and lower bounds for single Dirichlet and Neumann eigenvalues from two-sided asymptotically sharp bounds on averages. In view of this, we also recall an equivalent formulation of P\'olya's conjecture.




\subsection{The Berezin-Li-Yau method for Laplacian eigenfunctions}\label{sub:4:1}

An alternative way to obtain a bound of the form \eqref{bounds_neumann_eigenvalues} is through a generalization which we present here, of the method introduced by Berezin \cite{Berezin} and independently by Li and Yau \cite{LiYau} using the Fourier transform of eigenfunctions for the Dirichlet Laplacian. We extend this method to any Laplacian eigenfunction.

First, we recall that for any $f\in H^1_0(\Omega)$ the following facts hold:
\begin{equation}\label{facts}
  \int_{\mathbb{R}^d}|\hat{f}(\xi)|^2\,d\xi=\int_{\Omega}|f(x)|^2\,dx,\quad  \int_{\mathbb{R}^d}|\xi|^2|\hat{f}(\xi)|^2\,d\xi=\int_{\Omega}|\nabla f(x)|^2\,dx.
\end{equation}
(Here we denote still with $f$ the extension by zero of $f$ to $\mathbb R^d$).

Moreover, suppose that $\left\{f_j\right\}_{j=1}^{\infty}$ is an orthonormal basis of $L^2(\Omega)$ and let $\phi \in L^2(\Omega)$. Then
\begin{equation}\label{Parseval-id-phi-u}
  \sum_{j=0}^{\infty}|\widehat{\phi f_j}(\xi)|^2=(2\pi)^{-d}\int_{\Omega}|\phi(x)|^2\,dx.
\end{equation}

We are ready to state the following theorem.
\begin{theorem}\label{thm_neumann_preliminary}
  Suppose that $\left\{f_j\right\}_{j=1}^{\infty}$ is an orthonormal basis of $L^2(\Omega)$ and let $\phi \in L^2(\Omega)$ such that $\phi f_j \in H^{1}_{0}(\Omega)$ for all $j\in\mathbb N$. Then for all $k\in \mathbb{N}$ and all $R>0$ the following inequality holds:
\begin{equation}\label{BLY-ineq-all-R}
  \sum_{j=1}^{k}\int_{\Omega}|\nabla \phi f_j|^2\,dx\geq -\,\frac{2}{d+2}\,C_d^{-d/2}||\phi||^2_2R^{d+2}+R^2 \sum_{j=0}^{k}\int_{\Omega}| \phi f_j|^2\,dx.
\end{equation}
In particular,
\begin{equation}\label{BLY-ineq-optimal-R}
  \frac{d+2}{d}\sum_{j=0}^{k}\int_{\Omega}|\nabla \phi f_j|^2\,dx\geq C_d||\phi||_2^{-4/d}\bigg(\sum_{j=0}^{k}\int_{\Omega}| \phi f_j|^2\,dx\bigg)^{1+\frac{2}{d}}.
\end{equation}
\begin{proof}
Thanks to \eqref{facts} and \eqref{Parseval-id-phi-u}, for all $R>0$ we have
\begin{multline*}
\sum_{j=1}^{k}\int_{\Omega}|\nabla\phi f_j|^2dx=\sum_{j=1}^{k}\int_{\mathbb R^d}|\xi|^2|\widehat{\phi f_j}|^2d\xi\\
=\sum_{j=1}^{k}\int_{\mathbb R^d}(|\xi|^2-R^2)|\widehat{\phi f_j}|^2d\xi+R^2\sum_{j=1}^{k}\int_{\mathbb R^d}|\widehat{\phi f_j}|^2d\xi\\
\geq\sum_{j=1}^{k}\int_{|\xi|\leq R}(|\xi|^2-R^2)|\widehat{\phi f_j}|^2d\xi+R^2\sum_{j=1}^{k}\int_{\Omega}|\phi f_j|^2dx\\
\geq\int_{|\xi|\leq R}(|\xi|^2-R^2)\sum_{j=1}^{\infty}|\widehat{\phi f_j}|^2d\xi+R^2\sum_{j=1}^{k}\int_{\Omega}|\phi f_j|^2dx\\
=\int_{|\xi|\leq R}(|\xi|^2-R^2)(2\pi)^{-d}\|\phi\|_2^2d\xi+R^2\sum_{j=1}^{k}\int_{\Omega}|\phi f_j|^2dx\\
=-\frac{2}{d+2}C_d^{-d/2}\|\phi\|_2^2R^{d+2}+R^2\sum_{j=1}^k\int_{\Omega}|\phi f_j|^2dx.
\end{multline*}
This proves \eqref{BLY-ineq-all-R}. Now, the expression in the right side of \eqref{BLY-ineq-all-R} is optimized with respect to $R$ when
\begin{equation}\label{optimR}
R=\left(C_d\frac{\sum_{j=1}^{k}\int_{\Omega}|\phi f_j|^2dx}{\|\phi\|_2^2}\right)^{\frac{1}{2}}.
\end{equation}
Using \eqref{optimR} in \eqref{BLY-ineq-all-R} yields \eqref{BLY-ineq-optimal-R}. This concludes the proof.
\end{proof}
\end{theorem}

Theorem \ref{thm_neumann_preliminary} can be applied to lower bound the eigenvalues $\left\{\mu_j\right\}_{j=1}^{\infty}$ of $-\Delta_{\Omega}^N$ and obtain an inequality of the form \eqref{bounds_neumann_eigenvalues}, replacing the constant $c_d$ as necessary
by another constant depending only on the dimension. 

In fact, let us apply Theorem \ref{thm_neumann_preliminary} with $\left\{f_j\right\}_{j=1}^{\infty}=\left\{v_j\right\}_{j=1}^{\infty}$ (the Neumann eigenfunctions) as an orthonormal basis of $L^2(\Omega)$ and with $\phi=\phi_h$ defined by \eqref{phi-h-test-fct} with $f(p)=p$ (we choose $h<\bar h$, where $\bar h$ is defined by \eqref{barh}). Since $\Omega$ is of class $C^2$, $v_j\in H^2(\Omega)$ and $-\Delta v_j=\mu_j v_j$ is in $L^2(\Omega)$ for all $j\in\mathbb N$.
Hence \eqref{grad-phi-u-identity} holds with $u=v_j$. Therefore \eqref{grad-phi-u-identity} implies that
\begin{equation*}
\sum_{j=1}^{k}\int_{\Omega}|\nabla \phi_h v_j|^2=\sum_{j=1}^k\int_{\Omega}|\nabla\phi_h|^2 v_j^2 dx+\sum_{j=1}^k\mu_j\int_{\Omega}\phi_h^2 v_j^2dx.
\end{equation*}
From \eqref{BLY-ineq-optimal-R} it follows that
\begin{multline}\label{step1-final-proof}
\frac{d+2}{d}\sum_{j=1}^k\mu_j\int_{\Omega}\phi_h^2v_j^2dx\\
\geq C_d\|\phi_h\|_2^{-4/d}\left(\sum_{j=1}^k\int_{\Omega}\phi_h^2 v_j^2 dx\right)^{1+\frac{2}{d}}-\frac{d+2}{d}\sum_{j=1}^k\int_{\Omega}|\nabla \phi_h|^2 v_j^2 dx.
\end{multline}
Now, for the left side of \eqref{step1-final-proof}, from the fact that $\|\phi_h\|_{\infty}\leq 1$ and the normalization of $v_j$ we obtain
\begin{equation}\label{step2-final-proof}
\frac{d+2}{d}\sum_{j=1}^k\mu_j\int_{\Omega}\phi_h^2v_j^2dx\leq\frac{d+2}{d}\sum_{j=1}^k\mu_j.
\end{equation}
We next observe that from $\|\phi_h\|_{\infty}\leq 1$ it follows that $\|\phi_h\|_2^{4/d}\leq |\Omega|^{2/d}$. Hence, for the first term in the right side of \eqref{step1-final-proof}, we have
\begin{multline}\label{step3-final-proof}
C_d\|\phi_h\|_2^{-4/d}\left(\sum_{j=1}^k\int_{\Omega}\phi_h^2 v_j^2 dx\right)^{1+\frac{2}{d}}\\
\geq \frac{C_d}{|\Omega|^{2/d}}\left(\sum_{j=1}^k\int_{\Omega}v_j^2 dx-\sum_{j=1}^k\int_{\Omega}(1-\phi_h^2) v_j^2 dx\right)^{1+\frac{2}{d}}\\
\geq C_d\frac{k^{1+\frac{d}{2}}}{|\Omega|^{2/d}}\left(1-\frac{d+2}{d}\frac{1}{k}\sum_{j=1}^k\int_{\Omega}(1-\phi_h^2) v_j^2 dx\right)\\
=C_d\frac{k^{1+\frac{d}{2}}}{|\Omega|^{2/d}}-C_d\frac{d+2}{d}\frac{k^{\frac{d}{2}}}{|\Omega|^{2/d}}\sum_{j=1}^k\int_{\Omega}(1-\phi_h^2) v_j^2 dx.
\end{multline}
From \eqref{step2-final-proof} and \eqref{step3-final-proof} we obtain the following inequality
\begin{equation*}
\frac{1}{k}\sum_{j=1}^k\mu_j\\
\geq\frac{d}{d+2}C_d\left(\frac{k}{|\Omega|}\right)^{2/d}-\int_{\Omega}\left(C_d\left(\frac{k}{|\Omega|}\right)^{2/d}(1-\phi_h^2)+|\nabla\phi_h|^2\right)\frac{1}{k}\sum_{j=1}^k v_j^2dx.
\end{equation*}
Since $|1-\phi_h|\leq 1$, $|\nabla\phi_h|=1/h$ and $1-\phi_h$ and $\nabla\phi_h$ are supported on 
$\omega_h$, that for all $0<h<\bar h$, this inequality implies that
\begin{multline}\label{step00-finl-proof}
\frac{1}{k}\sum_{j=1}^k\mu_j
\geq\frac{d}{d+2}C_d\left(\frac{k}{|\Omega|}\right)^{2/d}\\
-\left(C_d\left(\frac{k}{|\Omega|}\right)^{2/d}+\frac{1}{h^2}\right)h|\partial\Omega|\left(\max_{x\in\omega_h}\frac{1}{k}\sum_{j=1}^k v_j^2(x)\right)-R_k(h),
\end{multline}
where
\begin{equation}\label{remainder_N}
R_k(h)=\left(C_d\left(\frac{k}{|\Omega|}\right)^{2/d}+\frac{1}{h^2}\right)(|\omega_h|-h|\partial\Omega|)\left(\max_{x\in\omega_h}\frac{1}{k}\sum_{j=1}^k v_j^2(x)\right).
\end{equation}
To optimize the second summand in the last line of \eqref{step00-finl-proof} with respect to $h$ we take $h$ in \eqref{step00-finl-proof} and \eqref{remainder_N} as
$$
h=h(k):=C_d^{-1/2}\left(\frac{k}{|\Omega|}\right)^{-1/d}.
$$
Assuming that
\begin{equation*}
k\geq \frac{2}{d+2}C_d^{-d/2}|\Omega|\max\left\{\bar h^{-d},(2/3)^{d}\max_{\substack{x\in\partial\Omega\\0\leq h\leq\bar h/2}}\left|\sum_{i=1}^{d-1}\frac{h\kappa_i(x)}{1-h\kappa_i(x)}\right|^{d}\right\},
\end{equation*}
we immediately verify that $0<h(k)<\bar h$ and moreover, by the same arguments used in the proof of Theorem \ref{thm_bounds_neumann_eigenvalues}, that
\begin{equation}\label{est_sp_fcn}
\max_{x\in\omega_h(k)}\sum_{j=1}^k v_j^2(x)\leq c_d'\frac{k}{|\Omega|},
\end{equation}
where $c_d'$ is a constant depending only on the dimension.
We plug \eqref{est_sp_fcn} into \eqref{step00-finl-proof} and obtain
\begin{equation*}
\frac{1}{k}\sum_{j=1}^k\mu_j\geq\frac{d}{d+2}C_d\left(\frac{k}{|\Omega|}\right)^{2/d}-2c_d'C_d^{1/2}\frac{|\partial\Omega|}{|\Omega|}\left(\frac{k}{|\Omega|}\right)^{1/d}-R(k),
\end{equation*}
where
$$
R(k):=R_k(h(k))
$$
and
$$
|R(k)|\leq 2c_d'C_d\left(\frac{k}{|\Omega|}\right)^{\frac{2}{d}}\frac{||\omega_{h(k)}|-h(k)|\partial\Omega||}{|\Omega|}.
$$
Exactly as in the proof of (i) in Theorem \ref{smooth_1} we observe that there exists a constant $C$ depending only on $\Omega$ such that $|R(k)|\leq C$, and  that
\begin{equation*}
\lim_{k\rightarrow+\infty}R(k)=\frac{(d-1)c_d'}{16|\Omega|}\left|\int_{\partial\Omega}\mathcal H(x)d\sigma(x)\right|.
\end{equation*}




\subsection{Asymptotically Weyl-sharp bounds on eigenvalues}\label{sub:4:2}
We conclude the paper with a general remark. Assume that $\Omega$ is such that
\begin{equation}\label{dirichlet_ineq_1}
\frac{1}{k}\sum_{j=1}^k\lambda_j\geq\frac{d}{d+2}C_d\left(\frac{k}{|\Omega|}\right)^{\frac{2}{d}}
\end{equation}
and
\begin{equation}\label{dirichlet_ineq_2}
\frac{1}{k}\sum_{j=1}^k\lambda_j\leq\frac{d}{d+2}C_d\left(\frac{k}{|\Omega|}\right)^{\frac{2}{d}}+A \left(\frac{k}{|\Omega|}\right)^{\frac{1}{d}}+B 
\end{equation}
for some constants $A,B$ independent of $k$, for all $k\geq k_0$. Then, for all  $k\geq k_0$
\begin{multline}\label{dirichlet_ineq_1_2}
\lambda_k\geq C_d\left(\frac{k}{|\Omega|}\right)^{\frac{2}{d}}\\
-\left(\frac{3}{d+2}C_d|\Omega|^{-\frac{2}{d}}+2A|\Omega|^{-\frac{1}{d}}\right)k^{\frac{3}{2d}}\\
+\frac{d+1}{d}A\left(\frac{k}{|\Omega|}\right)^{\frac{1}{d}}-\frac{3A}{2d}|\Omega|^{-\frac{1}{d}}k^{\frac{1}{2d}}-B,
\end{multline}
and
\begin{multline}\label{dirichlet_ineq_2_2}
\lambda_{k+1}\leq C_d\left(\frac{k}{|\Omega|}\right)^{\frac{2}{d}}\\
+\left(\frac{3}{d+2}C_d|\Omega|^{-\frac{2}{d}}+2A|\Omega|^{-\frac{1}{d}}\right)k^{\frac{3}{2d}}\\
+\frac{d+1}{d}A\left(\frac{k}{|\Omega|}\right)^{\frac{1}{d}}+\left(\frac{3A}{2d}|\Omega|^{-\frac{1}{d}}+2B\right)k^{\frac{1}{2d}}+B.
\end{multline}
In particular, for all $k\geq k_0$,
\begin{multline}\label{modulus_dirichlet}
\left|\lambda_k-C_d\left(\frac{k}{|\Omega|}\right)^{\frac{2}{d}}-\frac{d+1}{d}A\left(\frac{k}{|\Omega|}\right)^{\frac{1}{d}}\right|\\
\leq\left(\frac{3}{d+2}C_d|\Omega|^{-\frac{2}{d}}+2A|\Omega|^{-\frac{1}{d}}\right)k^{\frac{3}{2d}}+\left(\frac{3A}{2d}|\Omega|^{-\frac{1}{d}}+2B\right)k^{\frac{1}{2d}}+B.
\end{multline}
Inequalities \eqref{dirichlet_ineq_1_2} and \eqref{dirichlet_ineq_2_2} follow from \eqref{dirichlet_ineq_1} and \eqref{dirichlet_ineq_2} by observing that
$$
\lambda_k\geq\frac{1}{l}\sum_{j=k-l+1}^k\lambda_j
$$
and
$$
\lambda_{k+1}\leq\frac{1}{l}\sum_{j=k+1}^{k+l}\lambda_j,
$$
and by choosing $l\in\mathbb N$ such that
$$
l=k^{1-\frac{1}{2d}}+b
$$
with $b\in\left[-\frac{1}{2},\frac{1}{2}\right]$. In particular, with this choice, $\dfrac{1}{2}k^{1-\dfrac{1}{2d}}\leq l\leq \dfrac{3}{2}k^{1-\dfrac{1}{2d}}$ and $k-1\leq l\leq k+1$.

We note that the remainder estimate in \eqref{modulus_dirichlet} is not good since the power $k^{\frac{3}{2d}}$ is bigger than $k^{\frac{1}{d}}$. An analogous result holds if we replace Dirichlet eigenvalues $\lambda_j$ by Neumann eigenvalues $\mu_j$ (and reversing inequalities in \eqref{dirichlet_ineq_1},\eqref{dirichlet_ineq_2},\eqref{dirichlet_ineq_1_2} and \eqref{dirichlet_ineq_2_2}).

We observe that in this paper we have established inequalities of the form \eqref{dirichlet_ineq_2} for Dirichlet and Neumann eigenvalues (with ``$\leq$'' replaced by ``$\geq$'' in the case of Neumann eigenvalues) for different classes of domains in $\mathbb R^d$. Importantly, our estimates, and hence $A,B$ and $k_0$ in \eqref{modulus_dirichlet}, depend explicitly on the geometry of the domain, in particular on the (Hausdorff) measure of the boundary and the volume of the tubular neighbourhood about the boundary. For domains of class $C^2$ or convex domains the dependence is much more explicit. In view of this, we recall that P\'olya's conjecture
$$
\mu_{k+1}\leq C_d\left(\frac{k}{|\Omega|}\right)^{\frac{2}{d}}\leq\lambda_k
$$
is equivalent to
$$
\lim_{k\rightarrow+\infty}\frac{\lambda_k^*}{k^{\frac{2}{d}}}=\lim_{k\rightarrow+\infty}\frac{\mu_k^*}{k^{\frac{2}{d}}}=C_d,
$$
where
$$
\lambda_k^*=\inf\left\{\lambda_k:|\Omega|=1\right\}
$$
and
$$
\mu_k^*=\sup\left\{\mu_k:|\Omega|=1\right\}.
$$
We refer to \cite[Corollary 2.2]{colbois2014} for the proof of the equivalence. It is now clear that having inequalities of the form \eqref{modulus_dirichlet} with explicit geometric constants for quite broad classes of domains could prove useful in view of P\'olya's conjecture if  information on the geometry of optimizers of $\lambda_k$ and $\mu_k$ for large $k$ would be available (for example, $C^2$ smoothness, uniform boundedness of the surface measure and of the integral of the mean curvature).  This is, however, a notoriously
hard problem. For a few results in this direction see \cite{bucur2012,bucur2015}. We also refer to \cite{Antunes2012} and references therein for numerical optimization of low Dirichlet eigenvalues.

\section*{Acknowledgements}
The first and second authors are grateful to the \'Ecole Polytechnique
F\'ed\'erale de Lausanne for hospitality that supported this collaboration. The second author is also grateful to the \'Ecole Polytechnique
F\'ed\'erale de Lausanne since part of the work was done while he was post-doc there. The second author is member of the Gruppo Nazionale per l'Analisi Matematica, la
Probabilit\`a e le loro Applicazioni (GNAMPA) of the Istituto Nazionale di Alta Matematica (INdAM).





\end{document}